\documentclass[11pt]{article}

\usepackage{mathrsfs,amssymb,amsmath,amsfonts}
\usepackage{amsthm}
\usepackage{bbm}
\usepackage{tipa}
\usepackage{enumerate}
\usepackage{color}
\usepackage[titletoc,title]{appendix}
\usepackage{cite}
\usepackage[hyphens]{url}
\usepackage[bookmarks=false,hidelinks]{hyperref}

\usepackage{stmaryrd}

\usepackage{arydshln}
\usepackage{indentfirst}
\def\sqr#1#2{{\vcenter{\vbox{\hrule height.#2pt
              \hbox{\vrule width.#2pt height#1pt \kern#1pt \vrule width.#2pt}
              \hrule height.#2pt}}}}

\def\3n{\negthinspace \negthinspace \negthinspace }
\def\2n{\negthinspace \negthinspace }
\def\1n{\negthinspace }

\def\dbE{\mathbb{E}}
\def\dbF{\mathbb{F}}

\def\dbP{\mathbb{P}}
\def\dbQ{\mathbb{Q}}
\def\dbR{\mathbb{R}}

\def\={\buildrel \triangle \over =}

\def\ds{\displaystyle}

\def\ns{\noalign{\ss}}
\def\a{\alpha}
\def\b{\beta}
\def\g{\gamma}

\def\si{\sigma}
\def\t{\tau}
\def\f{\varphi}

\def\i{\infty}
\def\G{\Gamma}
\def\D{\Delta}

\def\O{\Omega}

\def\cA{{\cal A}}
\def\cB{{\cal B}}

\def\cF{{\cal F}}

\def\cH{{\cal H}}

\def\cS{{\cal S}}

\def\cU{{\cal U}}
\def\cV{{\cal V}}

\def\no{\noindent}

\def\ss{\smallskip}
\def\ms{\medskip}

\def\q{\quad}
\def\qq{\qquad}

\def\limsup{\mathop{\overline{\rm lim}}}
\def\liminf{\mathop{\underline{\rm lim}}}

\def\esssup{\mathop{\rm esssup}}
\def\essinf{\mathop{\rm essinf}}

\def\wt{\widetilde}

\def\cd{\cdot}

\def\les{\leqslant}
\def\ges{\geqslant}

\def\({\Big (}
\def\){\Big )}
\def\[{\Big[}
\def\]{\Big]}
\def\bde{\begin{definition}}
\def\ede{\end{definition}}
\def\be{\begin{equation}}
\def\bel{\begin{equation}\label}
\def\ee{\end{equation}}
\def\bt{\begin{theorem}}
\def\et{\end{theorem}}
\def\bc{\begin{corollary}}
\def\ec{\end{corollary}}
\def\bl{\begin{lemma}}
\def\el{\end{lemma}}
\def\bp{\begin{proposition}}
\def\ep{\end{proposition}}
\def\bas{\begin{assumption}}
\def\eas{\end{assumption}}
\def\br{\begin{remark}}
\def\er{\end{remark}}
\def\ba{\begin{array}}
\def\ea{\end{array}}
\def\ed{\end{document}}

\def\square#1{\vbox{\hrule\hbox{\vrule height#1%
     \kern#1\vrule}\hrule}}
\def\rectangle#1#2{\vbox{\hrule\hbox{\vrule height#1%
     \kern#2\vrule}\hrule}}

\font\tenbb=msbm10 \font\sevenbb=msbm7 \font\fivebb=msbm5

\newfam\bbfam
\scriptscriptfont\bbfam=\fivebb \textfont\bbfam=\tenbb
\scriptfont\bbfam=\sevenbb

\def\md{{\mathrm d}}

\newtheorem{assumption}{Assumption}

\newtheorem{theorem}{Theorem}[section]
\newtheorem{lemma}[theorem]{Lemma}
\newtheorem{definition}[theorem]{Definition}
\newtheorem{proposition}[theorem]{Proposition}
\newtheorem{corollary}[theorem]{Corollary}
\newtheorem{remark}[theorem]{Remark}

\theoremstyle{definition}

\newenvironment{keywords}{{\bf Key words: }}{}

\numberwithin{equation}{section}

\newcommand{\udots}{\mathinner{\mskip1mu\raise1pt\vbox{\kern7pt\hbox{.}}
\mskip2mu\raise4pt\hbox{.}\mskip2mu\raise7pt\hbox{.}\mskip1mu}}

\begin{document}

\title{Infinite-Horizon Non-Autonomous Zero-Sum Stochastic Recursive Differential Games and HJBI Equations\footnote{This work is supported by  the National Key R\&D Program of China (No. 2023YFA1009002),  the
National Natural Science Foundation of China (No. 12371443) and the Changbai
Talent Program of Jilin Province.}}

\author{Sheng Huang\footnote{School of Mathematics and Statistics, Northeast Normal University, Changchun 130024, P. R. China; {E-mail: huangsheng9910@163.com}}\qq
Qingmeng Wei\footnote{Corresponding author. School of Mathematics and Statistics, Northeast Normal University, Changchun 130024, P. R. China; E-mail:  weiqm100@nenu.edu.cn}
}

\maketitle

\begin{abstract}

In this paper, we study an infinite horizon non-autonomous stochastic recursive
differential game. To this end, we  first establish well-posedness and stability results for
BSDEs with a time-dependent discount factor and a possibly unbounded random
terminal time. The generator $f$ is allowed to be non-uniformly bounded at the origin, namely,
$ 
|f(t,0,0)|\les \beta_1(t)+\beta_2,
$ $t\in[0,\infty),$ $\dbP\text{-a.s.},
$ 
with $\beta_1\in L^1(0,\infty)$ and $\beta_2\ges0$.
We then formulate a two-person zero-sum stochastic recursive differential game
on the infinite horizon, where the drift, diffusion, generator and discount
factor may depend explicitly on time. The lower and upper value functions are
defined through Elliott--Kalton nonanticipative strategies and BSDE recursive
payoffs. By finite horizon approximation, BSDE stability estimates and viscosity
solution arguments, we prove that both the lower and upper value functions are
deterministic and are the unique bounded viscosity solutions of their
corresponding non-autonomous HJBI equations.
Finally, the time-homogeneous case is recovered as a special case. Using the
uniqueness of the non-autonomous HJBI equation, rather than a probabilistic
shift argument, we show from the PDE viewpoint that the value functions of the
autonomous system are independent of the initial time and solve  the
corresponding stationary HJBI equations in the viscosity sense.
\end{abstract}



\section{Introduction}

 Stochastic control and stochastic differential games provide powerful frameworks
for studying dynamic decision-making problems under uncertainty. In stochastic
control, a single decision-maker seeks to optimize a performance criterion,
whereas stochastic differential games model strategic interactions among
several decision makers.
Since the pioneering works in these fields, it has been well understood that the
dynamic programming principle (DPP) naturally leads to Hamilton--Jacobi--Bellman (HJB)
equations in stochastic control problems, and to
Hamilton--Jacobi--Bellman--Isaacs (HJBI) equations in zero-sum stochastic
differential games;  
see, for instance,
\cite{Krylov-1980,Fleming-Soner-2006,Yong-Zhou-1999, Pham-2009} for stochastic control
problems, and
\cite{Nisio-1988,Fleming-Souganidis-1989,Swiech-1996} for stochastic
differential games.
Since the seminal work of Pardoux and Peng \cite{Pardoux-Peng-1990}, backward
stochastic differential equations (BSDEs) have become a fundamental tool in
stochastic control, stochastic differential games, mathematical finance, and
related fields. When the payoff functional of a control or game problem is
defined through a BSDE, the associated value function is naturally linked to a
nonlinear HJB or HJBI equation. This connection has been extensively studied in
finite horizon stochastic control and differential games; see, among others,
\cite{Peng-1997, Pardoux-1998, Buckdahn-Li-2008,Li-Peng-2009,Buckdahn-Hu-Li-2011,
Li-Wei-2014,Li-Wei-2015,Li-Li-Wei-2021, Wei-Yu-2021}.

In many applications, however, the decision-making process does not necessarily
terminate at a prescribed deterministic time. The terminal time may be
determined by a stopping rule, as in optimal stopping problems, or by a default
event, as in credit-risk models; see, for instance,
\cite{Bielecki-Rutkowski-2004}. Alternatively, the problem may be formulated
over an infinite horizon, as in long-term investment, risk-sensitive control,
and economic differential games; we refer to
\cite{Fleming-McEneaney-1995,Fleming-Sheu-2002,
Dockner-Jorgensen-Long-Sorger-2000,Peskir-Shiryaev-2006}
for related studies.
When the payoff is described in a recursive way, such random or unbounded time
horizons naturally lead to BSDEs on random or infinite time intervals. Compared
with the finite horizon case, these BSDEs require additional assumptions to
ensure well-posedness, since the terminal time may be unbounded or may even
take the value infinity. Typical assumptions involve suitable integrability,
monotonicity, dissipativity, and discounting conditions. BSDEs with random
terminal times and related infinite horizon problems have been studied in
various settings.

Early work on this topic goes back to Peng \cite{Peng-1991}, where BSDEs with
random terminal times appeared in the probabilistic representation of
quasilinear PDEs. Darling and Pardoux \cite{Darling-Pardoux-1997} subsequently
established well-posedness results for such BSDEs and applied them to
semilinear elliptic PDEs. For a broader account of related BSDE--PDE
connections, we refer to Pardoux \cite{Pardoux-1998}.
Briand and Hu \cite{Briand-Hu-1998} further developed
stability results for such BSDEs and used them to study the homogenization of
systems of semilinear elliptic PDEs. Royer \cite{Royer-2004} studied
one-dimensional BSDEs with random terminal times driven by monotone generators,
extending the analysis from strictly monotone generators to merely monotone
ones, and investigated the corresponding links with elliptic PDEs. Quadratic
BSDEs with random terminal times, including the infinite horizon case, were
studied by Briand and Confortola \cite{Briand-Confortola-2008}, who related
them to elliptic PDEs in infinite-dimensional spaces. More recently, Lin et al. \cite{Lin-Ren-Touzi-Yang-2020} established a well-posedness
theory for second-order BSDEs with random terminal times.
From the infinite horizon viewpoint, Hu and Tessitore
\cite{Hu-Tessitore-2007} studied BSDEs and their applications to elliptic PDEs
in Hilbert spaces and optimal control problems.
In the recursive setting, Li and Zhao \cite{Li-Zhao-2019} investigated
vanishing-discount limits for nonexpansive stochastic control systems, and
Buckdahn, Li and Zhao \cite{Buckdahn-Li-Zhao-2021} extended this problem to
stochastic differential games, characterizing the limit values through HJBI
equations and dynamic programming arguments. Luo, Li and Wei
\cite{Luo-Li-Wei-2025} studied infinite horizon stochastic recursive control
problems with jumps and characterized the value function as a viscosity solution
of the associated HJB integro-differential equation.
Taken together, these works provide the BSDE estimates and PDE tools needed for
the analysis of recursive control and differential games on random or infinite
horizons. However, most of these studies are carried out in time-homogeneous or
constant-discount settings.

 Non-autonomous control and game problems arise naturally in economic and
financial applications, where the dynamics, costs, or discount mechanism may
vary with time. Deterministic infinite horizon non-autonomous optimal control
problems have been studied, for instance, by  Baumeister, Leit\~ao and Silva
\cite{Baumeister-Leitao-Silva-2007}, and Basco and Frankowska
\cite{Basco-Frankowska-2019}.
In the stochastic setting, however, the literature on 
infinite horizon non-autonomous problems is still rather limited. Wei and Yong
\cite{Wei-Yong-2025} recently studied an infinite horizon time-inconsistent
stochastic control problem, in which a non-autonomous time-consistent stochastic
control problem on an infinite horizon plays a key role. The value function of
this auxiliary problem is characterized by the associated time-dependent HJB
equation in the classical-solution framework.
Although this auxiliary control problem is not
recursive and does not involve BSDE payoffs, it naturally points to the
importance of developing a viscosity-solution framework for infinite horizon
non-autonomous problems in the stochastic setting.

 These two lines of research motivate the present work, in which we
study an infinite horizon non-autonomous stochastic recursive differential
game. 
Such a formulation calls for a BSDE framework adapted to time-dependent
coefficients and time-dependent discounting. We therefore first establish
well-posedness and stability results for BSDEs on random, possibly unbounded,
time intervals under this non-autonomous structure. In particular, the generator
is allowed to satisfy a non-uniform bound at the origin, for instance
$$
|f(t,0,0)|\les \beta_1(t)+\beta_2,\quad \dbP\mbox{-a.s.},
\qquad
\beta_1\in L^1(0,\infty),\quad \beta_2\ges 0.
$$
Our analysis adapts the truncation and approximation ideas developed for BSDEs
with random terminal times in
\cite{Darling-Pardoux-1997, Pardoux-1998, Briand-Hu-1998, Royer-2004}.
Although the argument relies on standard BSDE techniques, we record the main
estimates obtained along the way, since they will be used repeatedly in the
subsequent analysis of the game.

With this BSDE framework in hand, we proceed to the associated game problem.
The lower and upper value functions are formulated by using Elliott--Kalton type
nonanticipative strategies against admissible controls; see, for instance,
\cite{Fleming-Souganidis-1989,Buckdahn-Li-2008}. 
Since the coefficients and the discount factor may depend explicitly on time,
the value functions depend on both the initial time and the initial state, and
the associated equations are genuinely  infinite horizon non-autonomous  HJBI
equations.

We adopt a PDE approximation approach.  Instead of deriving the
DPP directly on the unbounded time
interval,   we
approximate the infinite horizon game by truncated finite horizon games with
zero terminal condition at the truncation time. The corresponding finite horizon
value functions solve finite horizon HJBI equations in the viscosity sense. Letting the truncation horizon tend to infinity and using BSDE stability
estimates, we identify the   infinite horizon non-autonomous  value functions   as the unique bounded viscosity solutions of the associated
infinite horizon HJBI equations. This PDE approximation avoids a direct dynamic
programming analysis on the unbounded time interval.

Moreover, our study also explains how the time-homogeneous case is recovered
from the non-autonomous formulation. When the coefficients are time-independent
and the discount factor is constant, the uniqueness of viscosity solutions
implies that the value functions are independent of the initial time, and the
associated HJBI equations reduce to stationary ones.
This shows that the non-autonomous formulation is consistent with the usual
stationary formulation of time-homogeneous infinite horizon control and game
problems, where the value functions depend only on the initial state and solve
stationary HJB/HJBI equations, refer to
\cite{Hu-Tessitore-2007, Li-Zhao-2019, Buckdahn-Li-Zhao-2021, Luo-Li-Wei-2025}, etc.

The rest of this paper is organized as follows. Section 2 establishes
well-posedness and stability estimates for BSDEs on random, possibly unbounded,
time intervals with a time-dependent discount factor. Section 3 is devoted to the infinite horizon non-autonomous zero-sum
stochastic recursive differential game. We introduce Elliott--Kalton type
strategies, define the lower and upper value functions through BSDE recursive
payoffs, and then characterize these value functions, by means of finite
horizon approximations, as the unique bounded viscosity solutions of the
associated non-autonomous HJBI equations. The
time-homogeneous case is then recovered as a special case, yielding the
corresponding stationary HJBI equations.

\section{BSDEs with Time-Dependent Discounting and Random Terminal Times}

We first collect some notation   used throughout the paper.
Let us begin with the underlying filtered probability space
$(\Omega,\mathcal{F},\mathbb{F},\mathbb{P})$,
equipped with a $d$-dimensional standard Brownian motion
$B=\{B_t\}_{t\ges0}$, where
$\mathbb{F}=\{\mathcal{F}_t\}_{t\ges0}$
is the natural filtration generated by $B$ and augmented by all $\dbP$-null sets.

For each $t\ges 0$, let $\mathcal S_t$ denote the set of all
$\mathbb F$-stopping times $\tau$ such that $t\les \tau\les\infty$,
$\mathbb P$-a.s., and set $\mathcal S:=\mathcal S_0$. For
$t\ges0$ and $\tau\in\mathcal S_t$, define  the random interval
$ \ds
\llbracket t,\tau\rrbracket
:=
\big\{(s,\omega)\in [0,\infty)\times\Omega:
t\les s\les \tau(\omega)\big\}.
$ 
Then, we introduce the following spaces.
\begin{itemize}
\item $L^1(t,\infty;\dbR)$: the set of all Borel measurable functions
$\varphi:[t,\infty)\to\mathbb{R}$ such that
$$\ds
\lVert\varphi\rVert_{L^1}:=\int_t^\infty |\varphi(s)|\,\md s<\infty.
$$

\item $L^\infty_{\mathcal{F}_\tau}(\mathbb{R})$: the set of all
$\mathcal{F}_\tau$-measurable $\mathbb{R}$-valued random variables $\xi$ such that
$$
\lVert\xi\rVert_{\i}:=\esssup_{\omega\in\Omega}|\xi(\omega)|<\infty.
$$

\item $L^\infty_{\mathbb{F}}(t,\tau;\mathbb{R})$: the set of all
$\mathbb{F}$-progressively measurable $\mathbb{R}$-valued processes
$\f=\{\f_s\}_{s\ges t}$ such that
$
\esssup\limits_{( s,\omega)\in\llbracket t,\tau\rrbracket }|\f_s(\omega)|<\infty.
$

\item $\cH^2_{\mathrm{loc}}(t,\tau;\mathbb{R}^d)$: the set of all
$\mathbb{F}$-progressively measurable $\mathbb{R}^d$-valued processes
$\f=\{\f_s\}_{s\ges t}$ such that, for every $T>t$,
$ \ds
\mathbb{E}\int_t^{T\wedge\tau}|\f_s|^2\,\md s<\infty.
$

     \item  For a subset $D$ of a  Euclidean space, $C_b(D)$, $\mathrm{LSC}\,(D)$ and
$\mathrm{USC}\,(D)$ denote the spaces of bounded continuous, lower semicontinuous and
upper semicontinuous real-valued functions on $D$, respectively.

\end{itemize}
 We shall use the notation
$ 
\mathbb E_s^{\mathbb Q}[\cdot]
:=
\mathbb E^{\mathbb Q}[\cdot\mid\mathcal F_s],
$ 
for any probability measure $\mathbb Q$ and  $s\ges 0$.

%


In what follows, we study a class of BSDEs with time-dependent discounting
and a random terminal time. The generator is allowed to have a time-dependent
growth bound, which is not necessarily uniform in time.

More precisely, let $\tau\in\cS$, and let $\xi\in L^\infty_{\mathcal F_\tau}(\mathbb R)$ be a
terminal random variable. We consider the following BSDE:
\begin{equation}
\label{A.2-tau}
\left\{\2n
\begin{aligned}
&Y_{s\wedge\tau}
 =
Y_{T\wedge\tau}
+
\int_{s\wedge\tau}^{T\wedge\tau}
\big(f(r, \rho(r)Y_r,Z_r)-\rho(r)Y_r\big)\,\md r
-
\int_{s\wedge\tau}^{T\wedge\tau} Z_r\,\md B_r,
\quad 0\les s\les T<\infty,
\\
&Y_{\tau} =\xi ,\quad \text{on }\{\tau<\infty\}.
\end{aligned}
\right.
\end{equation}
On $\{\tau=\infty\}$, every discounted terminal term is understood to
be zero; that is, for every finite time $r\ges0$,
$ 
\Gamma_{r,\tau}\xi=0$ on $\{\tau=\infty\},
$  
where 
\bel{Gamma}\ds
\Gamma_{t,s}:=\exp\(-\int_t^s \rho(r)\,\md r\),\q
 s\ges t\ges 0.
 \ee
The generator   $f:\O\times [0,\i)\times \dbR\times\dbR^d\to\dbR$   and the discount  factor $\rho:[0,\i)\to[0,\i)$  are assumed to satisfy the following conditions.
\begin{itemize}
\item[\bf(A1)]  For each $(y,z)\in\mathbb R\times\mathbb R^d$, $f(\cdot,\cdot,y,z)$ is   $\dbF$-progressively measurable, and the map $t\mapsto f(t,y,z)$ is continuous. Moreover, there exist constants $L_y,L_z\ges0$ such that,   for all $t\ges 0$,  $y_1,y_2\in \mathbb{R}$, and $z_1,z_2\in \mathbb{R}^d$,
$$
|f(t,y_1,z_1)-f(t,y_2,z_2)|
\les L_y|y_1-y_2|+L_z|z_1-z_2|,\ \dbP\mbox{-a.s.}
$$
Moreover, for all $t\ges 0$, $y_1,y_2\in \mathbb{R}$, and $z\in \mathbb{R}^d$,
$
\bigl(f(t,y_1,z)-f(t,y_2,z)\bigr)(y_1-y_2)\les 0,\ \dbP\mbox{-a.s.}
$

\item[\bf(A2)] There exist a measurable function $\beta_1:[0,\infty)\to[0,\infty)$ with $\beta_1\in L^1(0,\infty)$, and a constant $\beta_2\ges 0$ such that
$
|f(t,0 ,0)|\les \beta_1(t)+\beta_2,
$
for all $t\in [0,\infty)$, $\dbP$-a.s.

\item[\bf(A3)] The discount factor $\rho:[0,\infty)\to [0,\i)$ is continuous and there exists a constant $ \rho_0>0$ such that
$
\rho(s)\ges  \rho_0, $ for all $s\ges 0.
$

\end{itemize}

  First, we establish an a priori estimate for the BSDE \eqref{A.2-tau}.

\begin{proposition} \label{Pro-A.3}\sl
Assume  $\rho$ satisfies {\bf(A3)}.  For $i=1,2$, let $\xi_i\in L^{\i}_{\cF_\t}(\mathbb R)$ and $(Y^i,Z^i)\in L^\infty_{\mathbb{F}}(0,\t;\mathbb{R})\times \cH^2_{\mathrm{loc}}(0,\t;\mathbb{R}^d)$   be the solutions of BSDE \eqref{A.2-tau} associated with $(f_i,\xi_i) $,  respectively, where $f_i$ satisfies {\bf (A1)} and {\bf (A2)} with the same constants.
Then  there exists a probability measure $\mathbb Q$, locally equivalent to
$\mathbb P$, such that, for all $s\ges0$,
$$
|Y_{s\wedge\tau}^1-Y_{s\wedge\tau}^2| 
\les
 \mathbb E_{s\wedge\tau}^{\mathbb Q}
\[
\Gamma_{s\wedge\tau,\tau} 
|\xi_1-\xi_2| \mathbf 1_{\{\tau<\infty\}}
+
\int_{s\wedge\tau}^{\tau}
\Gamma_{s\wedge\tau,r}|\delta f_r| \,\md r
\].
$$
 Moreover, there exists a constant $C=C(\rho_0)>0$, independent of
$\tau$ and $\xi_i$, such that
$$
\mathbb E_{s\wedge\tau}^{\mathbb Q}
\[
\int_{s\wedge\tau}^{\tau}
\Gamma_{s\wedge\tau,r}|Z_{r}^1-Z_{r}^2|^2\,\md r
\]
\les
C\mathbb E_{s\wedge\tau}^{\mathbb Q}
\[
\Gamma_{s\wedge\tau,\tau}
|\xi_1-\xi_2|^2\mathbf 1_{\{\tau<\infty\}}
+
\int_{s\wedge\tau}^{\tau}
\Gamma_{s\wedge\tau,r}|\delta f_r|^2\,\md r
\],
$$
where
$ 
\delta f_r
:=
f_1(r,\rho(r)Y_r^2,Z_r^2)
-
f_2(r,\rho(r)Y_r^2,Z_r^2).
$

\end{proposition}

\begin{proof}
Set
$
\Delta Y:=Y^1-Y^2,$ $\Delta Z:=Z^1-Z^2,
$
Then, for any $T>s\ges 0$,
\bel{1-Delta-Y}
\Delta Y_{s\wedge\tau}
=
\Delta Y_{T\wedge\tau}
+
\int_{s\wedge\tau}^{T\wedge\tau}
\big(I_r^y+I_r^z+\delta f_r-\rho(r)\Delta Y_r\big)\,\md r
-
\int_{s\wedge\tau}^{T\wedge\tau}\Delta Z_r\,\md B_r,\ 0\les s\les T<\infty,
\ee
where
$
I_r^y
:=
f_1(r,\rho(r)Y^1_r,Z^1_r)
-
f_1(r,\rho(r)Y^2_r,Z^1_r),
$ $
I_r^z
:=
f_1(r,\rho(r)Y^2_r,Z^1_r)
-
f_1(r,\rho(r)Y^2_r,Z^2_r),
$
 $r\ges 0$.

For $r\ges 0$, define
$$
\gamma_r
:=
\begin{cases}
\dfrac{I_r^z}{|\Delta Z_r|^2}\,\Delta Z_r, & \Delta Z_r\neq 0,\\[1ex]
0, & \Delta Z_r=0.
\end{cases}
$$
By {\bf (A1)}, $|\gamma_r |\les L_z$, for all $r\ges 0$.
Then for each fixed $T>0$, define a probability measure $\mathbb Q^T$ on
$\mathcal F_{T }$ by
$$
\frac{\md \mathbb Q^T}{\md \mathbb P}\Big|_{\mathcal F_{T }}
=
\exp\(
\int_0^{T\wedge\tau }\gamma_r\,\md B_r
-\frac12\int_0^{T\wedge\tau}|\gamma_r|^2\,\md r
\).
$$
which is a true martingale on each finite interval. By Girsanov's theorem, the process
$ \ds
B_s^{\mathbb Q^T}
:=
B_s-\int_0^{s\wedge\tau}\gamma_r\,\md r,
$ $ 0\les s\les T,
$
is a $d$-dimensional Brownian motion under $\mathbb Q^T$.
Moreover, the family $(\mathbb Q^T)_{T>0}$ is consistent, in the sense that for $0<S<T$,
$
\mathbb Q^T\big|_{\mathcal F_{S}}=\mathbb Q^S.
$
Hence there exists a unique probability measure $\mathbb Q$ on
$
\mathcal F_\i:=\sigma\(\bigcup_{T>0}\mathcal F_{T }\)
$
such that
$
\mathbb Q\big|_{\mathcal F_{T }}=\mathbb Q^T,$ $ \forall T>0.
$
In particular, for each fixed $T>0$, the process
$ \ds
B_s^{\mathbb Q}
:=
B_s-\int_0^{s\wedge\tau}\gamma_r\,\md r,
 $ $ 0\les s\les T,
$
is a $d$-dimensional Brownian motion under $\mathbb Q$ on $[0,T]$.
 Therefore, in what follows, we work under the probability measure $\mathbb Q$ and write $ \ds
B_s^{\mathbb Q}
=
B_s-\int_0^{s\wedge\tau}\gamma_r\,\md r,
$ $s\ges0.
$

Therefore,  \eqref{1-Delta-Y}  can be rewritten as
\bel{Delta-Y-1}
\Delta Y_{s\wedge\tau}
=
\Delta Y_{T\wedge\tau}
+
\int_{s\wedge\tau}^{T\wedge\tau}
\bigl(I_r^y +\delta f_r-\rho(r)\Delta Y_r\bigr)\,\md r
-
\int_{s\wedge\tau}^{T\wedge\tau}\Delta Z_r\,\md B_r^ {\mathbb Q }, \ 0\les s\les T<\infty.
\ee
Applying Tanaka's formula to
$\Gamma_{{s\wedge\tau},r}|\Delta Y_r|$ on $\llbracket s\wedge\tau,T\wedge\tau\rrbracket$, we get
\bel{Delta-Y-Tana}
\begin{aligned}
|\Delta Y_{s\wedge\tau}|
&= \mathbb E_{s\wedge\tau}^{\mathbb Q}
 \[
\Gamma_{{s\wedge\tau},T\wedge\tau}|\Delta Y_{T\wedge\tau}|
+
\int_{s\wedge\tau}^{T\wedge\tau}
\Gamma_{{s\wedge\tau},r}\operatorname{sgn}(\Delta Y_r)
(I_r^y+\delta f_r)\,\md r
\\
&\qquad\qq
-
\int_{s\wedge\tau}^{T\wedge\tau}
\Gamma_{{s\wedge\tau},r}\operatorname{sgn}(\Delta Y_r)
\Delta Z_r\, \md B_r^{\mathbb Q}
-
\int_{s\wedge\tau}^{T\wedge\tau}
\Gamma_{{s\wedge\tau},r}\,\md L_r^0(\Delta Y)\] \\
&\les
\mathbb E_{s\wedge\tau}^{\mathbb Q}
\[
\Gamma_{{s\wedge\tau},{T\wedge\tau}}|\Delta Y_{T\wedge\tau}|
+
\int_{s\wedge\tau}^{T\wedge\tau}
\Gamma_{{s\wedge\tau},r}|\delta f_r|\,\md r
\],
\end{aligned}
\ee
where $L^0(\Delta Y)$ denotes the local time of $\Delta Y$ at zero and
$\operatorname{sgn}(0)=0$.

Applying It\^o's formula to $\Gamma_{s\wedge\tau,r}|\Delta Y_r|^2$ on $\llbracket s\wedge\tau,T\wedge\tau\rrbracket$
and taking conditional expectation under $\mathbb Q$, we get
\bel{Delta-Y-1f}\ba{ll}
\ns\ds |\Delta Y_{s\wedge\tau}|^2
+
\mathbb E_{s\wedge \t}^{\mathbb Q}\[
\int_{s\wedge\tau}^{T\wedge\tau}\Gamma_{{s\wedge\tau},r}  |\Delta Z_r|^2\,\md r+ \int_{s\wedge\tau}^{T\wedge\tau}
\rho(r)\Gamma_{{s\wedge\tau},r}|\Delta Y_r|^2\,\md r
\] \\
\ns\ds
=
\mathbb E_{s\wedge \t}^{\mathbb Q} \[
\Gamma_{{s\wedge\tau},T\wedge\tau} |\Delta Y_{T\wedge\tau}|^2
+
2\int_{s\wedge\tau}^{T\wedge\tau}\Gamma_{{s\wedge\tau},r} \Delta Y_r (I_r^y+\delta f_r)\,\md  r\] \\
\ns\ds \les\mathbb E_{s\wedge \t}^{\mathbb Q } [
\Gamma_{{s\wedge\tau},T\wedge\tau} |\Delta Y_{T\wedge\tau}|^2
+
\frac1{\rho_0}\int_{s\wedge\tau}^{T\wedge\tau}\Gamma_{{s\wedge\tau},r}  |\delta f_r|^2\,\md r
+
\int_{s\wedge\tau}^{T\wedge\tau}\rho(r)\Gamma_{{s\wedge\tau},r} |\Delta Y_r|^2\,\md  r
\],
\ea\ee
where we have used {\bf(A1)} and the Young's inequality.
Therefore,  there exists a constant $C=C(\rho_0)>0$ such that
\bel{Delta-Y-in}
\mathbb E_{s\wedge \t}^{\mathbb Q} \[
\int_{s\wedge\tau}^{T\wedge\tau}\Gamma_{{s\wedge\tau},r} |\Delta Z_r|^2\,\md r
\]
\les C\,
\mathbb E_{s\wedge \t}^{\mathbb Q } \[
\Gamma_{{s\wedge\tau},T\wedge\tau}  |\Delta Y_{T\wedge\tau}|^2
+
\int_{s\wedge\tau}^{T\wedge\tau}\Gamma_{{s\wedge\tau},r}  |\delta f_r|^2\,\md r
\].
\ee

Finally, since $
\Delta Y_{T\wedge\tau}
=
\mathbf 1_{\{\tau\les T\}}(\xi_1-\xi_2)
+
\mathbf 1_{\{\tau>T\}}\Delta Y_T,
$ and the boundedness of $Y_1$, $Y_2$, for $\a=1,$ $2$,
it follows that
$$\ba{ll}
\ns\ds
\Gamma_{{s\wedge\tau},T\wedge\tau}  |\Delta Y_{T\wedge\tau}|^\a
\les
 \Gamma_{{s\wedge\tau},\tau} |\xi_1-\xi_2|^\a\mathbf 1_{\{\tau\les T\}}
+
 \Gamma_{{s },T} |\Delta Y_T|^\a\mathbf 1_{\{\tau>T\}}\\
\ns\ds \les
 \Gamma_{{s\wedge\tau},\tau}  |\xi_1-\xi_2|^\a\mathbf 1_{\{\tau\les T\}}
+
Ce^{- \rho_0(T-s)}\mathbf 1_{\{\tau>T\}} ,\q \dbP\mbox{-a.s.}.
\ea$$
By dominated convergence, for $\a=1,2$, we conclude that
$$
\dbE^{\dbQ}_{s\wedge \t}\big[\Gamma_{{s\wedge\tau},T\wedge\tau}  |\Delta Y_{T\wedge\tau}|^\a\big]
\to
\dbE^{\dbQ}_{s\wedge \t}\big[\Gamma_{{s\wedge\tau},\tau}  |\xi_1-\xi_2|^\a\mathbf 1_{\{\tau<\infty\}}\big] ,\q \mbox {as } T\to\i.
$$
Consequently, letting $T\to\infty$ in \eqref{Delta-Y-Tana} and  \eqref{Delta-Y-in}, we obtain the desired results.
\end{proof}

By Proposition \ref{Pro-A.3}, we obtain the following stability result of \eqref{A.2-tau} with respect to terminal times.

\begin{corollary}
\label{Cor-terminal-time-stability}\sl
Assume that {\bf(A1)}--{\bf(A3)} hold. Let
$\tau_1,\tau_2\in\mathcal S$ satisfy
$ 
\tau_1\les \tau_2, $ $  \mathbb P\text{-a.s.}
$ 
For $i=1,2$, let $(Y^i,Z^i)$ be the solution of \eqref{A.2-tau} on
$\llbracket 0,\tau_i\rrbracket$ with the same generator $f$ and terminal
value
$ 
\xi_i\in L^\infty_{\mathcal F_{\tau_i}}(\mathbb R).
$ 
Then there exists a probability measure $\mathbb Q^{1,2}$, locally equivalent to
$\mathbb P$, such that, for all $s\ges 0$,
$$
|Y^1_{s\wedge\tau_1}-Y^2_{s\wedge\tau_1}|
\les
\mathbb E_{s\wedge\tau_1}^{\mathbb Q^{1,2}}
\left[
\Gamma_{s\wedge\tau_1,\tau_1}
|\xi_1-Y^2_{\tau_1}|\mathbf 1_{\{\tau_1<\infty\}}
\right].
$$
Moreover, there exists a constant $C=C(\rho_0)>0$, independent of
$\tau_1,\tau_2$ and the terminal values, such that
$$
\mathbb E_{s\wedge\tau_1}^{\mathbb Q^{1,2}}
\left[
\int_{s\wedge\tau_1}^{\tau_1}
\Gamma_{s\wedge\tau_1,r}  |Z_r^1-Z_r^2|^2\,\md r
\right]
\les
C
\mathbb E_{s\wedge\tau_1}^{\mathbb Q^{1,2}}
\left[
\Gamma_{s\wedge\tau_1,\tau_1} 
|\xi_1-Y_{\tau_1}^2|^2\mathbf 1_{\{\tau_1<\infty\}}
\right].
$$
\end{corollary}

\begin{proof}
By the flow property of BSDEs on random intervals, we have,  for all $0\les s\les T<\infty$,
$$
Y_{s\wedge\tau_1}^2
=
Y_{T\wedge\tau_1}^2
+
\int_{s\wedge\tau_1}^{T\wedge\tau_1}
\left(
f(r,\rho(r)Y_r^2,Z_r^2)-\rho(r)Y_r^2
\right)\md r
-
\int_{s\wedge\tau_1}^{T\wedge\tau_1}
Z_r^2\,\md B_r,
$$
with terminal value $Y_{\tau_1}^2$ at $\tau_1$ on $\{\t_1<\i\}$.
Therefore, on the common random interval
$\llbracket 0,\tau_1\rrbracket$, the two BSDEs for
$(Y^1,Z^1)$ and $(Y^2,Z^2)$ have the same generator and differ only in
their terminal values, namely $\xi_1$ and $Y_{\tau_1}^2$. Applying
Proposition \ref{Pro-A.3} to these two equations, we obtain the desired result. 
\end{proof}

Next, we study the well-posedness of BSDE \eqref{A.2-tau}.

 \begin{lemma}\label{LemmaA-1}\sl
Under {\rm {\bf(A1)}--{\bf(A3)}}, for any $\xi\in L^{\i}_{\cF_\t}(\mathbb R)$, there exists a unique pair
$
(Y ,Z )\in L^\infty_{\mathbb{F}}(0,\t;\mathbb{R})\times \cH^2_{\mathrm{loc}}(0,\t;\mathbb{R}^d)
$
solving the BSDE \eqref{A.2-tau}. Moreover,
\bel{Z-weighted-estimate-tau}
\operatorname*{ess\,sup}\limits_{(s,\omega)\in \llbracket 0,\tau\rrbracket }|Y_s(\omega)|\les M_1  \mbox { and }  \ \mathbb{E}\int_0^\tau \Gamma_{0,r}   |Z _r|^2\,\md r\les M_2 ,
\ee
  where
  $   \ds M_1:= \lVert\xi\rVert_\i+\lVert\beta_1\rVert_{L^1}+\frac{\beta_2}{\rho_0} , $ $  M_2:= 4M_1\( \lVert\beta_1\rVert_{L^1}
+\frac{\beta_2}{\rho_0}\)
+2M_1^2\(1 
+\frac{2L_z^2}{\rho_0}\).
$

\end{lemma}

\begin{proof}
 The proof follows the argument of    \cite[Theorem 2.1]{Royer-2004}, with suitable modifications to account for the time-dependent discount factor and the non-uniform bound on the generator $f$.
 Although the main argument is similar, we include a brief sketch here, since several properties of the truncated BSDEs established below will be used in the following section.

   For $n\in \mathbb{N}$, define  the stopping times
$
\tau_n:=\tau\wedge  n .$
  Consider the following truncated BSDE,
\bel{T-BSDE-tau}
\widetilde Y_s^n
= \xi_n+
\int_s^{\t_n}\(f(r, \rho(r)\widetilde Y_r^n,\widetilde Z_r^n)-\rho(r)\widetilde Y_r^n\)\,\md r
-\int_s^{\t_n} \widetilde Z_r^n\,\md B_r,
\quad s\in\llbracket 0,\tau_n\rrbracket ,
\ee
where $
\xi_n:=\xi\,\mathbf{1}_{\{\tau\les n\}} \in L^\i_{\mathcal F_{\tau_n}}(\dbR).
$
 For each $n\in\mathbb N$, since the stopping time $\tau_n $ is bounded, the truncated BSDE \eqref{T-BSDE-tau} admits a unique adapted solution $(\widetilde Y^n,\widetilde Z^n)$, see   \cite[Theorem 3.4]{Darling-Pardoux-1997}.
 For each $n\in\mathbb N$, define
$$
Y_s^n:=\widetilde Y_{s\wedge\tau_n}^n,\qquad
Z_s^n:=\widetilde Z_s^n\,\mathbf 1_{\llbracket 0,\tau_n\rrbracket }(s),
\qquad s\ges 0.
$$

\noindent
\emph{Step 1.} We show that $ (Y^n)$ is uniformly bounded in $n$.
By linearizing the generator in $z$ and applying a Girsanov transform as in Proposition  \ref{Pro-A.3},  one has  an equivalent
probability measure $\mathbb P^{n}$
and a Brownian motion
$
(B_s^n )_{s\in\llbracket 0,\tau_n\rrbracket }
$ under  $\mathbb P^n$ such that, \eqref{T-BSDE-tau}  can be rewritten as
$$
\widetilde Y_s^n
=\xi_n +
\int_s^{\t_n} \(f(r,  \rho(r)\widetilde Y_r^n,0)-\rho(r)\widetilde Y_r^n\)\,\md r
-\int_s^{\t_n}\widetilde  Z_r^n\,\md B_r^n,\q s\in\llbracket 0,\tau_n\rrbracket .
$$
Then, applying Tanaka's formula to
$\Gamma_{s,r}|\widetilde Y_r^n|$ on $\llbracket s,\tau_n\rrbracket$ under $\mathbb P^n$, we get
$$
\begin{aligned}
|\widetilde Y_s^n|
&=  \mathbb E_s^{\mathbb P^n}\[
\Gamma_{s,\tau_n}|\xi_n |
+
\int_s^{\tau_n}
\Gamma_{s,r}\operatorname{sgn}(\widetilde Y_r^n)
f(r,\rho(r)\widetilde Y_r^n,0)\,\md r  \\
&\qquad\qquad\quad
-
\int_s^{\tau_n}
\Gamma_{s,r}\operatorname{sgn}(\widetilde Y_r^n)
\widetilde Z_r^n\,\md B_r^n
-
\int_s^{\tau_n}\Gamma_{s,r}\,\md  L_r^0(\widetilde Y^n)\] \\
%
%
&\les  \mathbb E_s^{\mathbb P^n}
\[
\Gamma_{s,\tau_n}| \xi_n |
+
\int_s^{\tau_n}
\Gamma_{s,r}\,
\operatorname{sgn}(\widetilde Y_r^n)
|f(r,0,0)|\,\md r
\] \\
&\les
\lVert\xi\rVert_\infty
+
\mathbb E_s^{\mathbb P^n}
\[
\int_s^{\tau_n}
\Gamma_{s,r}\bigl(\beta_1(r)+\beta_2\bigr)\,\md r
\]  \\
&\les
\lVert\xi\rVert_\infty+\lVert\b_1\rVert _{L^1}
+\frac{\beta_2}{\rho_0}:=M_1,
\q  s\in \llbracket 0,\tau_n\rrbracket ,
\q \dbP^n\mbox{-a.s.} 
\end{aligned}
$$
where $L^0(\widetilde Y^n)$ denotes the local time of $\widetilde Y^n$ at
zero and $\operatorname{sgn}(0)=0$, and  we have used {\bf (A1)}   and the
nonnegativity of the local time term. Since $\mathbb P^n$ and $\mathbb P$ are equivalent on $\mathcal F_{\tau_n}$,
the estimate also holds $\mathbb P$-a.s.
 
 By definition,  for all $s\in \llbracket  \tau_n,\t\rrbracket $, $|Y_s^n|=|\widetilde Y_{\t_n}^n|=|\xi_n |\les M_1$, $\dbP$-a.s.
Therefore,
 together with the estimate on $\llbracket 0,\tau_n\rrbracket $, we conclude that
$ \ds
\operatorname*{ess\,sup}_{(s,\omega)\in\llbracket0,\tau\rrbracket}
|Y_s^n(\omega)|\les M_1,$ $ n\ges1 .
$ 
    \ms

 \no\emph{Step 2.} We prove, for every fixed $S\ges 0$, $(Y^n)_{n\ges1}$ converges uniformly on $\llbracket0,S\wedge\tau\rrbracket$.
Let $m>n$  and define
$
\widehat Y_s^{n,m}:=  Y_s^m-  Y_s^n,
$ $
\widehat Z_s^{n,m}:=   Z_s^m-  Z_s^n ,
$   for $s\in\llbracket 0,\tau\rrbracket $.
Still
by the techniques of    linearizing  in $z$ and the   Girsanov transform,   one can introduce an  equivalent    probability measure $\dbP^{n,m}$ and a $\dbP^{n,m}$-Brownian motion $B^{n,m}$ such that
$$
\widehat Y_s^{n,m}
=
\widehat Y_{\t_n}^{n,m}+\int_s^{\t_n} (\g_r^{n,m}-1)\rho(r)\widehat Y_r^{n,m}\,\md r
-\int_s^{\t_n} \widehat Z_r^{n,m}\,\md B_r^{n,m},\q s\in \llbracket 0,\tau_n\rrbracket ,
$$
where
$$
\g_r^{n,m}
:=
\left\{
\begin{aligned}
&
\frac{
f(r  ,\rho(r)  Y_r^{m}, Z_r^m)-f(r  ,\rho(r) Y_r^n,  Z_r^m)
}{
\rho(r)( Y_r^{m}-  Y_r^n)
},
\qquad  Y_r^{m}\neq   Y_r^n,
\\
& 0,\hskip 6.85cm   Y_r^{m}=   Y_r^n,
\end{aligned}
\right.
$$
and $\gamma_r^{n,m}\les 0$,  $|\g_r^{n,m}|\les L_y$ for all $r\in \llbracket 0,\tau_n\rrbracket $, $\dbP$-a.s.

Hence, by the representation formula for linear BSDEs,
$$\ba{ll}
\ns\ds 
|\widehat Y_s^{n,m}|
=\Big|
\mathbb E_s^{\mathbb P^{n,m}}
\[
\exp\(
\int_s^{\tau_n}
(\gamma_r^{n,m}-1)\rho(r)\,\md r
\)
\widehat Y_{\tau_n}^{n,m}
\]\Big|\les\mathbb E_s^{\mathbb P^{n,m}}
\[
\Gamma_{s,\tau_n}
|\widehat Y_{\tau_n}^{n,m}|
\] ,
\qquad s\in\llbracket 0,\tau_n\rrbracket .
\ea $$
Observe that
$\ds
\widehat Y_{\tau_n}^{n,m}
=
\widetilde Y_{\tau_n}^m-\xi_n
=
\mathbf 1_{\{\tau>n\}}\widetilde Y_n^m,
$ 
and  $|\widetilde Y_n^m|\les M_1$. Thus
$$
|\widehat Y_s^{n,m}|
\les
\mathbb E_s^{\dbP^{n,m}}\left[\Gamma_{s,\t_n}\mathbf 1_{\{\tau>n\}}| \widetilde Y_{n}^m|\right] 
\les
 M_1 e^{-\rho_0(n-s)},\q s\in \llbracket 0,\tau_n\rrbracket ,\  \dbP^{n,m}\mbox{-a.s.}
$$

Now fix $S\ges 0$ and assume that $n\ges S$. Then
$
\llbracket 0,S\wedge\tau\rrbracket\subset \llbracket0,\tau_n\rrbracket.
$
Hence
\bel{App-widehat Y-tau}
\sup_{s\in\llbracket 0,S\wedge\tau\rrbracket}|Y_s^m-Y_s^n|
\les
M_1 e^{-\rho_0(n-S)}
\to 0,
\qquad n\to\infty,\quad \dbP\text{-a.s.}
\ee
Thus, for every fixed $S\ges 0$, $(Y^n)_{n\ges1}$ is a Cauchy sequence uniformly on $\llbracket0,S\wedge\tau\rrbracket$, $\dbP$-a.s.  Consequently, there exists a continuous adapted process $\bar Y=(\bar Y_s)_{s\in\llbracket 0,\tau\rrbracket }$ such that, for every fixed $S\ges 0$,
\begin{equation}\label{Y^T-Y-2}
\sup_{s\in\llbracket 0,S\wedge\tau\rrbracket}|  Y_s^n-\bar Y_s|
\les
 M_1  e^{-\rho_0(n-S)}
\to 0,
\qquad n\to\infty,\quad \dbP\text{-a.s.}
\end{equation}

\noindent
\emph{Step 3.} We next study the convergence of $(Z^n)_{n\ges 1}$.
Fix $T>0$. let $m$, $n $ be integers such that $m> n\ges T$. Then, for $s\in[0,T]$,
\bel{Yhat-on-fixed-T}
\begin{aligned}
\widehat{Y}_{s\wedge\tau}^{n,m}
&=
\widehat{Y}_{T\wedge\tau}^{n,m}
+
\int_{s\wedge\tau}^{T\wedge\tau}
\Big(
f(r ,\rho(r)Y_r^m,Z_r^m)-f(r ,\rho(r)Y_r^n,Z_r^n)-\rho(r)\widehat{Y}_r^{n,m}
\Big)\,\md r
 \\
&\quad
-
\int_{s\wedge\tau}^{T\wedge\tau}\widehat{Z}_r^{n,m}\,\md B_r,\qq s\in [0,T].
\end{aligned}\ee

Applying It\^o's formula to $|\widehat Y_r^{n,m}|^2$ on $\llbracket0,T\wedge\tau\rrbracket$ and using the monotonicity condition in $y$ together
with the Lipschitz continuity in $z$, we obtain
$$
\mathbb{E}\int_0^{T\wedge\tau}|\widehat{Z}_r^{n,m}|^2\,\md r
\les
2\mathbb{E}|\widehat{Y}_{T\wedge\tau}^{n,m}|^2
+
4L_z^2\mathbb{E}\int_0^{T\wedge\tau}|\widehat{Y}_r^{n,m}|^2\,\md r\les
\(2+\frac{2L_z^2}{\rho_0}\)M_1^2 e^{-2\rho_0(n-T)},
$$
where \eqref{App-widehat Y-tau} is used.
Therefore, $(Z^n)_{n\ges 1}$ is a Cauchy sequence in $L_{\dbF}^2(0,T\wedge \t;\mathbb R^d)$ for every $T>0$.
Consequently, there exists a progressively measurable process $\bar Z\in \cH^2_{\mathrm{loc}}(0,\tau;\mathbb R^d)$  such that
$$
\lim_{n\to\infty}\mathbb E\int_0^{T\wedge \t} |Z_r^n-\bar Z_r|^2\,\md r=0,
\qquad \forall T>0.
$$

\noindent
\emph{Step 4.}  We now pass to the limit and identify the equation satisfied by $(\bar Y,\bar Z)$.
Fix $T>0$. For any integer  $n$ such that $n\ges T$, the truncated BSDE yields
$$
Y_{s\wedge \t}^n
=
Y_{T\wedge \t}^n
+\int_{s\wedge \t}^{T\wedge \t} \big(f(r  ,\rho(r)Y_r^n,Z_r^n)-\rho(r)Y_r^n\big)\,\md r
-\int_{s\wedge \t}^{T\wedge \t} Z_r^n\,\md B_r,
\qquad s\in[0,T].
$$
By    the  Lipschitz continuity
of $f$ in $(y,z)$,  and the dominated convergence theorem, letting $n\to\infty$, Steps 2 and 3 yield
$$
\bar Y_{s\wedge \t}
=
\bar Y_{T\wedge \t}
+\int_{s\wedge \t}^{T\wedge \t} \big(f(r  ,\rho(r)\bar Y_r,\bar Z_r)-\rho(r)\bar Y_r\big)\,\md r
-\int_{s\wedge \t}^{T\wedge \t} \bar Z_r\,\md B_r,
\qquad s\in[0,T].
$$
 Moreover, $|\bar Y_s|\les M_1$, $s\in[0,\t]$, $\dbP$-a.s., and
$ \ds
\mathbb E\int_0^{T\wedge\t} |\bar Z_r|^2\,\md r<\infty,$ $\forall\,T>0.
$

Furthermore, for every $n$ such that $n\ges T$, on the set $\{\tau\les T\}$ one has
$\tau_n=\tau$ and hence
$
Y_\tau^n=\xi_n =\xi.
$
Letting $n\to\infty$  and using the uniform convergence in Step 2, we deduce that
$$
\bar Y_\tau=\xi,
\qquad \text{on }\{\tau\les T\}.
$$
Since $T>0$ is arbitrary and $\ds
\{\tau<\infty\}=\bigcup_{N=1}^\infty \{\tau\les N\}
$, it follows that
$
\bar Y_\tau=\xi 
$ on $\{\tau<\infty\}.
$
Therefore,  $(\bar Y,\bar Z)$ solves the BSDE \eqref{A.2-tau}.

\ms
 
It remains to estimate $\bar Z$. Applying It\^o's formula to
$\Gamma_{0,s}|\bar Y_s|^2$ on   $\llbracket0,T\wedge\tau\rrbracket$, we get
$$
\begin{aligned}
&\mathbb E\left[
|\bar Y_0|^2
+
\int_0^{T\wedge\tau}\Gamma_{0,s}|\bar Z_s|^2\,ds
+
\int_0^{T\wedge\tau}\rho(s)\Gamma_{0,s}|\bar Y_s|^2\,ds
\right]
\\
&=
\mathbb E\left[
\Gamma_{0,T\wedge\tau}|\bar Y_{T\wedge\tau}|^2
+
2\int_0^{T\wedge\tau}
\Gamma_{0,s}\bar Y_s f(s,\rho(s)\bar Y_s,\bar Z_s)\,ds
\right]\\
&\les \mathbb E\[
\Gamma_{0,{T\wedge\tau}}M_1^2
 +
 \int_0^{T\wedge\tau} 2M_1\Gamma_{0,s} \big(\beta_1(s)+\beta_2 +L_z|\bar Z_s|\big) \,\md s\]\\
&\les \mathbb E\[
\Gamma_{0,{T\wedge\tau}} M_1^2
 +
 \int_0^{T\wedge\tau} 2M_1\Gamma_{0,s} \big(\beta_1(s)+\beta_2 \big) \,\md s+
\int_0^{T\wedge\tau}\Gamma_{0,s}\(\frac 12|\bar Z_s|^2+ 2M_1  ^2 L_z^2 \) \,\md s\].
\end{aligned}
$$
Therefore,
$$
\begin{aligned}
 & \mathbb E\int_0^{T\wedge\tau} \Gamma_{0,s} |\bar Z_s|^2\,\md s
\les\mathbb E\[
2\Gamma_{0,{T\wedge\tau}} M_1^2+  \int_0^{T\wedge\tau} \Gamma_{0,s} \(4M_1 (\beta_1(s)+\beta_2  )+4M_1^2 L_z^2   \)\,\md s \] \\
 &\les
2 M_1^2+  4M_1\lVert\beta_1\rVert_{L^1}
+\frac{4M_1\beta_2}{\rho_0} 
+\frac{4L_z^2}{\rho_0}M_1^2:= M_2,
\end{aligned}
$$
where we have used
$$
\int_0^{T\wedge\tau}  \Gamma_{0,s} \beta_1(s)\,\md s\les \lVert\beta_1\rVert_{L^1},
\
\int_0^{T\wedge\tau} \Gamma_{0,s} \,\md s\les \frac{1}{ \rho_0},\
\int_0^{T\wedge\tau}  \Gamma_{0,s} \rho(s)\,\md s
\les 1.
$$ 
 Letting $T\to\infty$, the monotone convergence theorem yields
\eqref{Z-weighted-estimate-tau}.

Finally, by applying Proposition \ref{Pro-A.3} to two solutions of BSDE \eqref{A.2-tau} with the same data, we immediately obtain the  uniqueness.
\end{proof}

Next, we apply Proposition  \ref{Pro-A.3} and Lemma \ref{LemmaA-1}  to establish the well-posedness and stability of a class of decoupled FBSDEs with random terminal times.
  In detail,
for any $(t,x)\in [0,\infty)\times \mathbb{R}^n$, $\t\in\cS_t$, consider the forward SDE
\begin{equation}\label{A.1-tau}
\left\{\2n
\begin{aligned}
&\md X_s^{t,x}  = b(s,X_s^{t,x})\,\md s+\sigma(s,X_s^{t,x})\,\md B_s,\qquad s\ges t,\\
&X_t = x.
\end{aligned}
\right.
\end{equation}
and  the BSDE
\begin{equation}
\label{A.2-tau-g}
\left\{\2n
\begin{aligned}
&Y_{s\wedge\tau}^{t,x}
 =
Y_{T\wedge\tau}^{t,x}
+
\int_{s\wedge\tau}^{T\wedge\tau}
\Big(g(r,X_r^{t,x},\rho(r)Y_r^{t,x},Z_r^{t,x})-\rho(r)Y_r^{t,x}\Big)\,\md r
\\
&\qq\q
-
\int_{s\wedge\tau}^{T\wedge\tau} Z_r^{t,x}\,\md B_r,
\quad t\les s\les T<\infty,
\\
&Y_{\tau} ^{t,x}=\psi(X_\tau^{t,x}) ,\quad \text{on }\{\tau<\infty\}.
\end{aligned}
\right.
\end{equation}
%
We assume that the coefficients $b:\O\times [0,\i)\times\dbR^n\to\dbR^n$, $\si:\O\times [0,\i)\times\dbR^n\to\dbR^{n\times d},
 $ $g:\O\times [0,\i)\times\dbR^n\times \dbR\times\dbR^d\to\dbR$, and $\psi:\O\times \dbR^n\to\dbR$ satisfy the following conditions.
\begin{itemize}
\item[\bf(A4)]  For each fixed $x\in\mathbb R^n$,  $b(\cdot,x)$ and $\sigma(\cdot,x)$ are $\dbF$-progressively measurable. $
b
$ and $
\sigma
$
are uniformly bounded, continuous in $t\in [0,\i)$, and uniformly Lipschitz continuous in $x\in\dbR^n$.

\item[\bf(A5)] For each fixed $(x,y,z)$,  $g(\cdot,x,y,z)$ is $\dbF$-progressively measurable.
 $
g
$
is continuous in $t$, and there exist constants $L_x,L_y,L_z\ges 0$ such that, for all $t\ges 0$, $x_1,x_2\in \mathbb{R}^n$, $y_1,y_2\in \mathbb{R}$, and $z_1,z_2\in \mathbb{R}^d$,
$$
|g(t,x_1,y_1,z_1)-g(t,x_2,y_2,z_2)|
\les
L_x|x_1-x_2|+L_y|y_1-y_2|+L_z|z_1-z_2|,\ \dbP\mbox{-a.s.}
$$
Moreover, for all $t\ges 0$, $x\in \mathbb{R}^n$, $y_1,y_2\in \mathbb{R}$, and $z\in \mathbb{R}^d$,
$$
\bigl(g(t,x,y_1,z)-g(t,x,y_2,z)\bigr)(y_1-y_2)\les 0,\ \dbP\mbox{-a.s.}
$$

\item[\bf(A6)] There exist a measurable function $\beta_1:[0,\infty)\to[0,\infty)$ with $\beta_1\in L^1(0,\infty)$, and a constant $\beta_2\ges 0$ such that
$
|g(t,x,0 ,0)|\les \beta_1(t)+\beta_2,
$
for all $(t,x)\in [0,\infty)\times \mathbb{R}^n$, $\dbP$-a.s.

  \item  [\bf(A7)] For each  $x$,  $\psi(x)$ is $\cF_\t$-measurable. Moreover,  there exist constants $B_\psi\ges 0$ and $L_\psi\ges0$    such that, for all  $x,x_1,x_2\in\mathbb R^n,$
$
|\psi(x)|\les B_\psi$, and $
|\psi(x_1)-\psi(x_2)|\les L_\psi|x_1-x_2|,
$ $\dbP$-a.s.

\item[{\bf(A8)}]  There exists a constant $\mu >-\rho_0$ such that, for all $t\ges 0$ and $x,\bar x\in\mathbb R^n$,
$$
 2\langle x-\bar x,\; b(t,x)-b(t,\bar x)\rangle
+  |\sigma(t,x)-\sigma(t,\bar x) |^2+ 2L_z|\si(t,x)-\si(t,\bar x)||x-\bar x|
\les -\mu |x-\bar x|^2, \ \dbP\mbox{-a.s.}
$$

\end{itemize}

 Under {\bf(A4)}, SDE   \eqref{A.1-tau}  admits a unique strong solution. Consequently, we obtain the following well-posedness result for BSDE \eqref{A.2-tau-g}, which follows immediately from  Lemma \ref{LemmaA-1}  applied to
$
\xi=\psi(X_\tau^{t,x})$ and $ f(r,y,z)=g(r,X_r^{t,x},y,z).
$

 \begin{lemma}\label{LemmaA-1-FBSDE}\sl
Under {\rm {\bf(A3)}--{\bf(A7)}}, for any $(t,x)\in [0,\infty)\times \mathbb{R}^n$  and any $\tau\in\mathcal S_t$, there exists a unique pair
$
(Y^{t,x},Z^{t,x})\in L^\infty_{\mathbb{F}}(t,\t;\mathbb{R})\times \cH^2_{\mathrm{loc}}(t,\t;\mathbb{R}^d)
$
solving \eqref{A.2-tau-g}. Moreover, for all $s\ges t$,
$$
|Y_s^{t,x}|
 \les M_3,
\ \dbP\mbox{-a.s., \  and }
 \q
\mathbb{E}\int_t^\tau \Gamma_{t,s}  |Z^{t,x}_s|^2\,\md s\les M_4 ,
$$
   where
$
 M_3:=B_\psi+\lVert\beta_1\rVert_{L^1}+\frac{\beta_2}{\rho_0} , $ $  M_4:= 4M_3\big(\lVert\beta_1\rVert_{L^1}+\frac{\beta_2}{\rho_0}\big)
+
2M_3^2\big(1+\frac{2L_z^2}{\rho_0}\big).
$
\end{lemma}

 Furthermore, by  Proposition \ref{Pro-A.3}, we obtain the Lipschitz continuity of $Y_t^{t,x}$ with respect to the initial state $x\in\mathbb R^n$ as follows.

 \begin{lemma}\label{LemmaA-2}\sl
Under {\rm {\bf(A3)}--{\bf(A8)}}, for any   $t\ges 0$,
  for all $x,\bar x\in \mathbb{R}^n$,
\bel{Y-Lip}
|Y_t^{t,x}-Y_t^{t,\bar x}|
\les 
\Big( 
L_\psi^2+\frac{L_x^2}{ \rho_0(\rho_0+\mu)}
\Big)^{1/2}
|x-\bar x|, \qquad \mathbb P\text{-a.s.} 
\ee
In particular, if $\psi\equiv 0$, then
\begin{equation}
|Y_t^{t,x}-Y_t^{t,\bar x}|
\les
 \frac{L_x}{\sqrt{ \rho_0(\rho_0+\mu)}}|x-\bar x|,\qquad \mathbb P\text{-a.s.}
\label{Yt-Lipschitz-tau-Phi0}
\end{equation}
\end{lemma}

\begin{proof}

 For simplicity, we write
$
\Delta X_s:=X_s^{t,x}-X_s^{t,\bar x},$ $
\Delta Y_s:=Y_s^{t,x}-Y_s^{t,\bar x},$ $
\Delta Z_s:=Z_s^{t,x}-Z_s^{t,\bar x},$ $ s\ges t.
$
 By considering $f_1(s,\rho(s)y,z):=f(s,X_s^{t,x}, \rho(s)y,z)$, $f_2(s,\rho(s)y,z):=f(s,X_s^{t,\bar x}, \rho(s)y,z)$
 in  the proof of  Proposition \ref{Pro-A.3}, from estimate \eqref{Delta-Y-1f}, we  obtain
\bel{BSDE-Gro}\ba{ll}
\ns\ds |\Delta Y_t|^2 +\dbE^{\mathbb Q}_t\Big[\int_t^{T\wedge\tau} \G_{t,r} |\Delta Z_r|^2\md r\Big]
 \les \mathbb E^{\mathbb Q}_t\Big[
 \Gamma_{t,T\wedge\tau}|\Delta Y_{T\wedge\tau}|^2
+\int_t^{T\wedge\tau}  \frac1{\rho_0}   \Gamma_{t,r} |I_r^x| ^2  \,\md r
\Big] \\
\ns\ds  \les \mathbb E^{\mathbb Q}_t\Big[
 \Gamma_{t,T\wedge\tau}|\Delta Y_{T\wedge\tau}|^2
+ \frac{L_x^2}{\rho_0} \int_t^{T\wedge\tau}   \Gamma_{t,r} |\D X_r| ^2  \,\md r
\Big] ,
\ea\ee
where
$ \ds
I_r^x:=g(r,X_r^{t,x},\rho(r)Y_r^{t,x},Z_r^{t,x})-g(r,X_r^{t,\bar x},\rho(r)Y_r^{t,x},Z_r^{t,x}).
$ Here, $\mathbb Q $ is the probability measure introduced in the proof of Proposition \ref{Pro-A.3}.

We now estimate $\Delta X$ under $\mathbb Q$. Under $\mathbb Q$, the process $\Delta X$ satisfies
$$\ba{ll}
\ns\ds
\md\Delta X_s
=
\big(
b(s,X_s^{t,x})-b(s,X_s^{t, \bar x})
+(\sigma(s,X_s^{t,x})-\sigma(s,X_s^{t, \bar x}))\gamma_s
\big)\,\md s\\
\ns\ds \qq\qq
+(\sigma(s,X_s^{t,x})-\sigma(s,X_s^{t, \bar x}))\,\md B_s^{\mathbb Q},\q s\in [t,T].
\ea$$
Applying It\^o's formula to $e^{ \mu(r-t)}|\Delta X_r|^2$ on $\llbracket t,s\wedge\tau\rrbracket$, we obtain
$$
\begin{aligned}
&\mathbb E_t^{\mathbb Q} \[e^{ \mu(s\wedge\tau-t)}|\Delta X_{s\wedge\tau}|^2\] = |x-\bar x|^2
+\mathbb E_t^{\mathbb Q} \[\int_t^{s\wedge\tau} e^{ \mu(r-t)}
\Bigl(2\langle \Delta X_r, b(r,X_r^{t,x})-b(r,X_r^{t,\bar x})\rangle
 \\
&\qquad\quad+2\langle \Delta X_r, (\sigma(r,X_r^{t,x})-\sigma(r,X_r^{t,\bar x}))\gamma_r\rangle
+|\sigma(r,X_r^{t,x})-\sigma(r,X_r^{t,\bar x})|^2+
 \mu|\Delta X_r|^2
\Bigr)\,\md r\]\\
&\les |x-\bar x|^2,
\end{aligned}$$
where we have used $|\gamma_r|\les L_z$ and  condition  {\bf(A8)}.
Consequently,
\bel{Delta X}
\mathbb E_t^{\mathbb Q} \left[e^{ \mu(\t-t)}|\Delta X_\t |^2\mathbf 1_{ \{\t\les T\}} \right]
\les  |x-\bar x|^2,
\qquad s\in [t,T  ].
\ee

By  {\bf(A7)}, we have
\bel{terminal-term-estimate}
\begin{aligned}
\dbE_t^{\mathbb Q} \left[ \Gamma_{t,T\wedge\tau}|\Delta Y_{T\wedge\tau}|^2\right]
&=
\dbE_t^{\mathbb Q} \left[
\Gamma_{t,T\wedge\tau} 
\Bigl(
\mathbf 1_{\{\tau\les T\}}(\psi(X_\tau^{t,x})-\psi(X_\tau^{t,\bar x}))
+\mathbf 1_{\{\tau>T\}}\Delta Y_T
\Bigr)^2
\right] \\
&\les 
L_\psi^2\,\dbE_t^{\mathbb Q} \left[\Gamma_{t,\tau} |\Delta X_\tau|^2\mathbf 1_{\{\tau\les T\}}\right]
+ \dbE_t^{\mathbb Q}\left[\Gamma_{t,T} |\Delta Y_T|^2\mathbf 1_{\{\tau>T\}}\right]\\
&\les  
L_\psi^2\,\dbE_t^{\mathbb Q} \left[e^{ \mu(\tau-t)}|\Delta X_\tau|^2\mathbf 1_{\{\tau\les T\}} \right]
+4|M_3|^2\cd  \dbE_t^{\mathbb Q}\left[\Gamma_{t,T} \mathbf 1_{\{\tau>T\}}\right],
\end{aligned}
\ee
where we have used $\mu>-\rho_0$, and the boundedness of $Y^{t,x}$ and $Y^{t,\bar x}$ established in Lemma \ref{LemmaA-1-FBSDE}.
Thus
$$
\limsup_{T\to\infty}\dbE_t^{\mathbb Q} \left[ \Gamma_{t,T\wedge\tau}|\Delta Y_{T\wedge\tau}|^2\right]
\les  L_\psi^2|x-\bar x|^2.
$$

Substituting \eqref{Delta X} and \eqref{terminal-term-estimate} into \eqref{BSDE-Gro}, we obtain
$$
|\Delta Y_t|^2
 \les
 \( 
L_\psi^2
+\frac{L_x^2}{\rho_0}\int_t^T e^{- \rho_0(r-t)}e^{- \mu(r-t)}\,\md  r
\)|x-\bar x|^2.
$$
Letting $T\to\infty$, we get
$
|\Delta Y_t|^2
\les
 \big( 
L_\psi^2+\frac{L_x^2}{ \rho_0(\rho_0+\mu)}
\big)|x-\bar x|^2.
$
 Then, \eqref{Y-Lip} is proved. When   $\psi\equiv 0$,
\eqref{Yt-Lipschitz-tau-Phi0} is obvious from \eqref{Y-Lip}.
\end{proof}

Furthermore, we have the following result for the infinite horizon BSDE.

\br\label{ReA-3} \sl
If $\tau\equiv\infty$, then the terminal condition
$Y_\tau^{t,x}=\psi(X_\tau^{t,x})$ on $\{\tau<\infty\}$ is void, and  \eqref{A.2-tau-g}   reduces to the following infinite horizon BSDE
 \bel{BSDE-A}\ba{ll}
 \ns\ds
Y_s^{t,x}
=
Y_T^{t,x}+\int_s^T \bigl(g(r,X_r^{t,x},\rho(r)Y_r^{t,x},Z_r^{t,x})-\rho(r)Y_r^{t,x}\bigr)\,\md r\\
\ns\ds \qq\q-\int_s^T Z_r^{t,x}\,\md B_r,
\quad \mbox{ for all }\, t\les s\les T<\infty.
\ea\ee
Therefore, by Lemma  \ref{LemmaA-1-FBSDE},   under {\rm \bf(A3)--(A6)}, it follows that \eqref{BSDE-A} admits a unique bounded adapted solution $(Y^{t,x},Z^{t,x})$. Moreover,
$
|Y_s^{t,x}|\les M_1^\infty,$ $ \forall s\ges t,$ $ \dbP$-a.s.,
and $ \ds \dbE\int_t^\infty \Gamma_{t,s} |Z_s^{t,x}|^2\,\md s\les M_2^\infty,
$
where
\bel{App-M-12-i}
M_1^\infty:=
 \lVert\beta_1\rVert_{L^1(0,\infty)}+\frac{\beta_2}{\rho_0} , \q M_2^\infty:= 4M_1^\infty\( \lVert\beta_1\rVert_{L^1}
+\frac{\beta_2}{\rho_0}\)
+2(M_1^\i)^2\(1
+\frac{2L_z^2}{\rho_0}\).\ee

Furthermore, if ${\bf (A8)}$ holds, then by    the special case $\psi\equiv0$  in  Lemma \ref{LemmaA-2}, for any  fixed $t\in [0,\infty)$ and all $x,\bar x\in \mathbb{R}^n$,
\begin{equation*}
|Y_t^{t,x}-Y_t^{t,\bar x}|
\les
 \frac{L_x}{\sqrt{ \rho_0(\rho_0+\mu)}}|x-\bar x|.
\label{Yt-Lipschitz}
\end{equation*}

 \er

We conclude this part with the following remarks. The infinite horizon BSDE in
Remark \ref{ReA-3} may be viewed as an uncontrolled counterpart of the
framework in \cite{Buckdahn-Li-Zhao-2021,Li-Zhao-2019}. In contrast to these
works, we allow a time-dependent discount factor $\rho(\cd)$ and replace the
uniform boundedness of the generator at the origin by
$$
|g(t,x,0,0)|\les \beta_1(t)+\beta_2,\qquad (t,x)\in [0,\infty)\times\mathbb R^n,
$$
with $\beta_1\in L^1(0,\infty)$ and $\beta_2\ges0$.
The Lipschitz condition of $g$ with respect to $y$ could be further relaxed by
 the approximation argument of \cite[Lemma 2.1]{Buckdahn-Li-Zhao-2021}. Since our main focus is the following
stochastic differential game problem,  we do not pursue this extension  here.

\section{Infinite Horizon Non-Autonomous Stochastic Recursive Differential Game}\label{sec:Non-autonomous SDG-MT}

Now we  study an infinite horizon non-autonomous stochastic recursive differential game.
Following the standard framework in differential games; see, for example,
\cite{Buckdahn-Li-2008}, we  first present the following definitions of admissible
controls and nonanticipative strategies on the  infinite horizon.
Throughout this section, $U$ and $V$ are assumed to be nonempty compact
metric spaces.

\begin{definition}\label{admissible controls-infinite horizon}\sl
Let $t\ges0$. Fix $u_0\in U$.
A process
$
u=\{u_s,\ s\ges t\}
$
is called an admissible control for Player $1$ on $[t,\infty)$, if its
extension
$ 
\bar u_s
:=
u_s\mathbf 1_{\{s\ges t\}}
+
u_0\mathbf 1_{\{0\les s<t\}},
$ $ s\ges0,
$ 
is $\mathbb F$-progressively measurable and takes values in $U$.
The set of all admissible controls for Player $1$ on $[t,\infty)$ is denoted
by $\mathcal U_{t,\infty}$.
We identify two controls $u_1$ and $u_2$ in $\mathcal U_{t,\infty}$ and write
$
u_1\equiv u_2
$
on $[t,\infty)$, if
$ 
u_1(s,\omega)=u_2(s,\omega),
$ $
\md s\otimes\md\mathbb P\text{-a.e. on }[t,\infty)\times\Omega .
$ 

Similarly, the admissible controls for Player $2$ on $[t,\infty)$ are defined
by replacing $U$, $u_0$, and $\mathcal U_{t,\infty}$ above with $V$, $v_0$,
and $\mathcal V_{t,\infty}$, respectively.
\end{definition} 
 
\begin{definition}\label{nonanticipative-strategies-infinite horizon}\sl
Let $t\ges0$. A nonanticipative strategy for Player $1$ on $[t,\infty)$ is a
mapping
 $
\alpha:\mathcal V_{t,\infty}\to \mathcal U_{t,\infty}
$ 
such that, for any $S\in\mathcal S_t$ and any
$v_1,v_2\in\mathcal V_{t,\infty}$, if
$ 
v_1\equiv v_2
$ on $\llbracket t,S\rrbracket,
$ 
then
$ 
\alpha(v_1)\equiv \alpha(v_2)
$ on $ \llbracket t,S\rrbracket.
$ 
The set of all nonanticipative strategies for Player $1$ on $[t,\infty)$ is
denoted by $\mathcal A_{t,\infty}$.

Similarly, a nonanticipative strategy for Player $2$ on $[t,\infty)$ is a
mapping
$ 
\beta:\mathcal U_{t,\infty}\to \mathcal V_{t,\infty},
 $
defined analogously. The set of all such nonanticipative strategies is denoted
by $\mathcal B_{t,\infty}$.
\end{definition}

For any  $(t,x)\in[0,\i)\times\dbR^n $  and admissible controls $u(\cd)\in \cU_{t,\i},$ $v(\cd)\in \cV_{t,\i}$, consider the controlled state equation
    \bel{state-NA}
\begin{cases}
\md X_s^{t,x;u,v} = b(s,X_s^{t,x;u,v}, u_s, v_s)\md s + \sigma(s,X_s^{t,x;u,v}, u_s, v_s)\md B_s, \q  s\ges  t, \\
X_t^{t,x;u,v} = x,
\end{cases}
\ee
and  the following infinite horizon BSDE
\begin{equation}\label{BSDE-NA}
  \ba{ll}
 \ns\ds Y_{s}^{t,x;u,v}
 =  Y_{T}^{t,x;u,v}
 + \int_s^T   \bigl(g (r, X_r^{t,x;u,v}, \rho(r)Y_{r}^{t,x;u,v}, Z_r^{t,x;u,v}  , u_r,v_r)-\rho(r)Y_{r}^{t,x;u,v}\bigr)  \md r\\
 \ns\ds
\qq\qq\q -\int_s^T  Z_r^{t,x;u,v}  \md  B_r, \q  \mbox{for all }t \les s \les T <\i.
    \ea
      \end{equation}
Note that the above forward-backward system is non-autonomous, in the sense that the coefficients $b$, $\sigma$, $g$, and the discount factor $\rho$ are allowed to depend explicitly on the time variable.
  The discount factor $\rho:[0,\infty)\to [0,\infty)$  is still the one in  {\bf(A3)}.
  We impose the   assumptions on the  coefficients
$
b:[0,\infty)\times \mathbb{R}^n\times U\times V \to \mathbb{R}^n
$, $
\sigma:[0,\infty)\times \mathbb{R}^n\times U\times V \to \mathbb{R}^{n\times d}$ and $
g:[0,\infty)\times \mathbb{R}^n\times \mathbb{R}\times \mathbb{R}^d\times U\times V \to \mathbb{R}
$ as follows.

\begin{description}
 \item[{\bf(H1)}] The coefficients
$
b
$ and $
\sigma
$
are uniformly bounded  on $[0,\infty)\times \mathbb{R}^n\times U\times V$,  and continuous in $(s,u,v)$, and uniformly Lipschitz continuous in $x$.

\item[{\bf(H2)}] The function
$
g
$
is continuous in $(s,u,v)$, and  there exist constants $L_x,L_y,L_z\ges 0$ such that, for all $s\ges 0$, $x_1,x_2\in\mathbb{R}^n$, $y_1,y_2\in\mathbb{R}$, $z_1,z_2\in\mathbb{R}^d$, $u\in U$, and $v\in V$,
$$
|g(s,x_1,y_1,z_1,u,v)-g(s,x_2,y_2,z_2,u,v)|
\les
L_x|x_1-x_2|+L_y|y_1-y_2|+L_z|z_1-z_2|.
$$
Moreover, for all $s\ges 0$, $x\in\mathbb{R}^n$, $y_1,y_2\in\mathbb{R}$, $z\in\mathbb{R}^d$, $u\in U$, and $v\in V$,
$$
\bigl(g(s,x,y_1,z,u,v)-g(s,x,y_2,z,u,v)\bigr)(y_1-y_2)\les 0.
$$

\item[{\bf(H3)}] There exist a measurable function $\beta_1:[0,\infty)\to[0,\infty)$ with
$
\beta_1\in L^1(0,\infty;\dbR^+),
$
and a constant $\beta_2\ges 0$ such that, for all $(s,x,u,v)\in [0,\infty)\times \mathbb{R}^n \times U\times V$,
$
|g(s,x,0,0,u,v)|\les \beta_1(s)+\beta_2.
$
\end{description}

Under {\bf(H1)}, for any $(t,x)\in [0,\infty)\times \mathbb R^n$ and any admissible control pair
$(u,v)\in \mathcal U_{t,\infty}\times \mathcal V_{t,\infty}$, the controlled state equation
\eqref{state-NA}  admits a unique strong solution, denoted by $X^{t,x;u,v}$. Furthermore, by Remark \ref{ReA-3}, if {\bf(A3)} and {\bf(H2)}--{\bf(H3)} also hold,
  then the backward equation \eqref{BSDE-NA}
admits a unique solution
$
(Y^{t,x;u,v},Z^{t,x;u,v})\in
L^\infty_{\mathbb F}(t,\infty;\mathbb R)\times \cH^2_{\mathrm{loc}}(t,\infty;\mathbb R^d),
$
which satisfies
\begin{equation}\label{4.5}
|Y_s^{t,x;u,v}|
\les M_1^\i, \ s\ges t,\ \dbP\mbox{-a.s.,  and } \ \
\mathbb E \int_t^\infty \Gamma_{t,s}\, |Z_s^{t,x;u,v}|^2\,\md s
\les M_2^\i.
\end{equation}
Here  $M_1^\i$  and $M_2^\i$ are defined as in \eqref{App-M-12-i}; in particular, $M_1^\i$ and $M_2^\i$ are independent of $(t,x)$ and $(u,v)$.

 If, in addition, the following condition holds:
\begin{description}
 \item[{\bf(H4)}] There exists a constant $\mu>-\rho_0$ such that, for all $s\ges 0$, $x,x'\in \mathbb R^n$,
$u\in U$, and $v\in V$,
$$\ba{ll}
\ns\ds
2\langle x-x',  b(s,x,u,v)-b(s,x',u,v)\rangle
+  |\sigma(s,x,u,v)-\sigma(s,x',u,v)|^2\\
\ns\ds
+2L_z |\sigma(s,x,u,v)-\sigma(s,x',u,v)| \cd|x-x'|
\les -\mu |x-x'|^2,
\ea$$

\end{description}
then  Remark \ref{ReA-3}  yields that, for any  $t\ges 0$, $x,x'\in  \mathbb R^n$
and any admissible control pair $(u,v)\in \mathcal U_{t,\infty}\times \mathcal V_{t,\infty}$,
\begin{equation}\label{Y-lip}
|Y_t^{t,x;u,v}-Y_t^{t,x';u,v}|\les
\frac{L_x}{\sqrt{\rho_0(\rho_0+\mu)}}|x-x'|, \quad \mathbb P\text{-a.s.}
\end{equation}

 With the above preparation, we can  rigorously formulate the infinite horizon non-autonomous stochastic recursive differential game associated with \eqref{state-NA} and \eqref{BSDE-NA}. And its lower and upper value functions are well defined, respectively, by
\bel{NA-uW}\underline{W}(t,x)=  \essinf_{\b \in \cB_{t,\i} } \esssup_{u \in \cU_{t,\i}  }   Y_{t}^{t,x;u,\b(u)} ,\q (t,x)\in [0,\i)\times\dbR^n,  \ee
and
\bel{NA-oW}\overline{W}(t,x)= \esssup_{\a \in \cA_{t,\i}   }   \essinf_{v \in \cV_{t,\i} }  Y_{t}^{t,x;\a(v),v}  ,\q (t,x)\in [0,\i)\times\dbR^n. \ee
 In what follows, we shall focus only on the lower value function $\underline{W}$; all  the corresponding results for the upper value function $\overline{W}$ can be established in a completely analogous manner.

Based on \eqref{4.5} and \eqref{Y-lip}, by the definitions of essential supremum and infimum, it follows that,
for all $t\in [0,\infty)$, $x,x'\in \mathbb R^n$,
\bel{W-bound}
|\underline{W}(t,x)|\les M_1^\i,  \q
|\underline{W}(t,x)-\underline{W}(t,x')|\les
\frac{L_x}{\sqrt{\rho_0(\rho_0+\mu)}}|x-x'| ,\q\dbP\mbox{-a.s.}
\ee

  \begin{remark}  \sl

  The estimates in \eqref{W-bound} are now obtained in the
$\mathbb P$-a.s. sense. Once the determinacy of $\underline W$ is established
in Theorem~\ref{Th-NA}, we identify $\underline W$ with its deterministic
version, and then \eqref{W-bound} holds pointwise.

  \end{remark}

Next, consider the HJBI equation
\begin{equation}\label{HJBI-NA-L}
\left\{\2n\ba{ll}
\ns\ds \partial_t   W(t,x)-\rho(t)  W(t,x)
+\underline H\bigl(t,x,(  W,D W,D^2 W)(t,x)\bigr)=0,
\qquad (t,x)\in [0,\infty)\times\mathbb R^n,\\
\ns\ds \lim_{T\to\infty}\Gamma_{0,T}  W(T,x)=0,
\qquad \text{uniformly in }x\in\mathbb R^n.
\ea\right.\end{equation}
where,  for 
$(t,x,y,p,A)\in [0,\infty)\times\mathbb R^n\times\mathbb R
\times\mathbb R^n\times\mathbb S^n$,     
$$
\underline H(t,x,y,p,A)
:=
\sup_{u\in U}\inf_{v\in V}
\Big\{
\frac12\operatorname{tr}
\big(\sigma\sigma^\top(t,x,u,v)A\big)
+b(t,x,u,v)\cdot p
+g\big(t,x,\rho(t)y,p\sigma(t,x,u,v),u,v\big)
\Big\}.
$$  
If, in addition, $\beta_2=0$, then the above boundary condition can be
strengthened to
\begin{equation}\label{BC-2}
\lim_{T\to\infty}  W(T,x)=0,
\qquad \text{uniformly in }x\in\mathbb R^n.
\end{equation}
 We shall identify the lower value function as a viscosity solution of the
associated HJBI equation.  
Throughout this section, viscosity solutions of   equation
\eqref{HJBI-NA-L} are understood in the standard parabolic sense, e.g., \cite{Crandall-Ishii-Lions-1992, Fleming-Soner-2006, Buckdahn-Li-2008}.

\begin{definition}\label{Def-viscosity-solution-NA}\sl
 A function $w\in C([0,\infty)\times\mathbb R^n)$ is called  
\begin{itemize}
\item[(i)]  a
viscosity subsolution of \eqref{HJBI-NA-L} if, for every
$(t_0,x_0)\in[0,\infty)\times\mathbb R^n$ and every
$\varphi\in C^{3}_{l,b}([0,\infty)\times\mathbb R^n)$ such that
$w-\varphi$ attains a local maximum at $(t_0,x_0)$ on
$[0,\infty)\times\mathbb R^n$, one has
$$
\partial_t\varphi(t_0,x_0)
-\rho(t_0)w(t_0,x_0)
+\underline H\bigl(t_0,x_0,w(t_0,x_0),
D\varphi(t_0,x_0),D^2\varphi(t_0,x_0)\bigr)
\ges 0,
$$
and, moreover,
$ \ds
\limsup_{T\to\infty}\sup_{x\in\mathbb R^n}
\Gamma_{0,T}w(T,x)\les 0 .
$ 

\item[(ii)]   a
viscosity supersolution of \eqref{HJBI-NA-L} if, for every
$(t_0,x_0)\in[0,\infty)\times\mathbb R^n$ and every
$\varphi\in C^{3}_{l,b}([0,\infty)\times\mathbb R^n)$ such that
$w-\varphi$ attains a local minimum at $(t_0,x_0)$ on
$[0,\infty)\times\mathbb R^n$, one has
$$
\partial_t\varphi(t_0,x_0)
-\rho(t_0)w(t_0,x_0)
+\underline H\bigl(t_0,x_0,w(t_0,x_0),
D\varphi(t_0,x_0),D^2\varphi(t_0,x_0)\bigr)
\les 0,
$$
and, moreover,
$ \ds
\liminf_{T\to\infty}\inf_{x\in\mathbb R^n}
\Gamma_{0,T}w(T,x)\ges 0 .
$ 

\item[(iii)] A function $w\in C_b([0,\infty)\times\mathbb R^n)$ is called a
viscosity solution of \eqref{HJBI-NA-L} if it is both a viscosity subsolution
and a viscosity supersolution. In this case, equivalently,
$ \ds
\lim_{T\to\infty}\Gamma_{0,T}w(T,x)=0,
$ uniformly in $x\in\mathbb R^n .
$ 
\end{itemize}

\end{definition}
 
If, in addition, $\beta_2=0$, then the boundary condition at infinity in
Definition \ref{Def-viscosity-solution-NA} may be replaced by the stronger
undiscounted boundary condition
$\ds 
\lim_{T\to\infty}w(T,x)=0,
$ uniformly in $x\in\mathbb R^n
$ 
for viscosity solutions, with the corresponding one-sided inequalities for
subsolutions and supersolutions.

\begin{theorem}\label{Th-NA}\sl
Under assumptions {\bf(A3)} and {\bf(H1)}--{\bf(H4)}, the lower value function
$\underline W$ is deterministic and belongs to
$C_b([0,\infty)\times\mathbb R^n)$. Moreover, $\underline W$ is the unique
viscosity solution of the HJBI equation \eqref{HJBI-NA-L}   in
$C_b([0,\infty)\times\mathbb R^n)$.
If, in addition, $\beta_2=0$, then $\underline W$ also satisfies the stronger
undiscounted boundary condition \eqref{BC-2}.
\end{theorem}

  \begin{proof}
  For any $T>t$, consider the finite horizon truncation of \eqref{BSDE-NA}  at  $T$ as follows,
 \begin{equation}\label{approximating BSDE-zero terimal condition}
  \ba{ll}
 \ns\ds Y_{s}^{T,t,x;u,v}
 =  \int_s^T \big(  g (r, X_r^{t,x;u,v},\rho(r) Y_{r}^{T,t,x;u,v}, Z_r^{T,t,x;u,v}  , u_r,v_r)-\rho(r) Y_{r}^{T,t,x;u,v}\big) \, \md r \\
  \ns\ds\qq\qq\qq -  \int_s^T  Z_r^{T,t,x;u,v}  \,\md  B_r,\q s\in [t,T].
    \ea
      \end{equation}
Under assumptions {\bf(H1)}--{\bf(H4)}, for each fixed $T>0$,  equation \eqref{approximating BSDE-zero terimal condition} admits a unique adapted solution. We then define the corresponding  lower value function by
\bel{WT-definition}W^T(t,x)=  \essinf_{\b \in \cB_{t,T} } \esssup_{u \in \cU_{t,T}   }   Y_{t}^{T,t,x;u,\b(u)}  ,\q (t,x)\in[0,T]\times \dbR^n. \ee
By the finite horizon result in \cite{Buckdahn-Li-2008}, $W^T$ is deterministic and is
the unique viscosity solution,  among continuous functions with polynomial
growth in the spatial variable,   of the following
HJBI equation,  
    \begin{equation}\label{approximating finite horizon HJBI equation}\3n\3n\left\{\3n
\ba{ll}
\ns\ds
\frac{\partial }{\partial t}W^T(t,x)  -\rho(t) W^T(t,x) +  \underline{H}\big(t,x,(W^T,DW^T,D^2W^T)(t,x)\big)  =0,\  (t,x)\in [0,T)\times \dbR^n,\\
\ns\ds W^T(T,x)  = 0,\q x\in\dbR^n.
  \ea  \right.  \end{equation}

Fix $S>0$. For any $T,T'>S$,  and $(t,x)\in[0,S]\times\mathbb R^n$,
similar to   \eqref{App-widehat Y-tau} and  \eqref{Y^T-Y-2} in the proof of Lemma  \ref{LemmaA-1}, there exists a constant $C>0$, independent of $(u,v)$, $(t,x)$ and $T,$ $T'$, such that
\bel{Sec4:Y^T-Y-1}\ba{ll}
\ns\ds
\sup_{s\in[t,S]}|Y_s^{T,t,x;u,v}-Y_s^{T',t,x;u,v}|
\les
Ce^{-\rho_0 ((T\wedge T')-S )}
 ,\q  \dbP\mbox{-a.s.},  \\
\ns\ds \sup_{s\in[t,S]}|Y_s^{T,t,x;u,v}-Y_s^{t,x;u,v}|
\les
Ce^{-\rho_0(T-S)}
 ,\q  \dbP\mbox{-a.s.}
\ea\ee

\no\emph{Step 1.} We first show that the lower value function $\underline W$ is deterministic, and belongs to $C_b([0,\i)\times\dbR^n)$,
and that $W^T$ converges to $\underline W$ uniformly on $[0,S]\times\mathbb R^n$,
for each $S>0$.
To prove this, we first observe that
$$
W^T(t,x)
=
\essinf_{\beta\in\mathcal B_{t,T}}
\esssup_{u\in\mathcal U_{t,T}}
Y_t^{T,t,x;u,\beta(u)}
=
\essinf_{\beta\in\mathcal B_{t,\infty}}
\esssup_{u\in\mathcal U_{t,\infty}}
Y_t^{T,t,x;u,\beta(u)},
\quad (t,x)\in[0,T]\times\mathbb R^n,
$$
whose proof is standard; see \cite[Lemma 3.13]{Li-Li-Wei-2021} for similar arguments. It follows from the fact that the finite horizon payoff
$Y_t^{T,t,x;u,\beta(u)}$ depends only on the restrictions of $u$ and $\beta(u)$ to $[t,T]$.
Indeed, every $\beta\in\mathcal B_{t,\infty}$ induces, by restriction, a strategy
$\beta^T\in\mathcal B_{t,T}$, while every $\beta^T\in\mathcal B_{t,T}$ can be extended to some
$\tilde\beta\in\mathcal B_{t,\infty}$. Since admissible controls are required only to be progressively  measurable,
and admissible strategies only to be nonanticipative, these restriction and extension operations are well defined.

Using the first estimate in \eqref{Sec4:Y^T-Y-1}, together with the definitions of essential
supremum and infimum, we have, for any $T,T'>S$ and
$(t,x)\in[0,S]\times\mathbb R^n$,
$$
|W^T(t,x)-W^{T'}(t,x)|
\les
C e^{-\rho_0((T\wedge T')-S)}.
$$
Hence $\{W^T\}_{T>S}$ is a Cauchy family uniformly on
$[0,S]\times\mathbb R^n$. Since $W^T$ is deterministic and continuous
for each $T>S$, there exists a deterministic continuous function
$w^S$ on $[0,S]\times\mathbb R^n$ such that
$$
\lim_{T\to\infty}
\sup_{(t,x)\in[0,S]\times\mathbb R^n}
|W^T(t,x)-w^S(t,x)|=0.
$$
Moreover, if $0<S_1<S_2$, then $w^{S_1}=w^{S_2}$ on
$[0,S_1]\times\mathbb R^n$. Since $S>0$ is arbitrary, these local limits
define a deterministic continuous function $w$ on
$[0,\infty)\times\mathbb R^n$.

Next, using the second estimate in \eqref{Sec4:Y^T-Y-1}, again together with the definitions
of essential supremum and infimum, we obtain, for any $T>S$ and
$(t,x)\in[0,S]\times\mathbb R^n$,
$$
|W^T(t,x)-\underline W(t,x)|
\les
C e^{-\rho_0(T-S)},\qquad \mathbb P\text{-a.s.}
$$
Letting $T\to\infty$ yields
$
\underline W(t,x)=w(t,x),$ $ \mathbb P\text{-a.s.}
$
Therefore, $\underline W(t,x)$ is deterministic  for each $(t,x)$. Identifying $\underline W$
with its deterministic continuous version $w$, we further obtain
$$
\lim_{T\to\infty}
\sup_{(t,x)\in[0,S]\times\mathbb R^n}
|W^T(t,x)-\underline W(t,x)|=0.
$$
Finally, by the boundedness estimate in \eqref{W-bound}, $\underline W\in C_b([0,\infty)\times\mathbb R^n)$.
This proves the desired statement.

 \ms 
 
 \no \emph{Step 2.}  We verify  the boundary conditions for $\underline{W}$.
By the boundedness of $\underline{W}$
 and {\bf(A3)},
we have
$$
\sup_{x\in\mathbb R^n}|\Gamma_{0,t}\, \underline{W}(t,x)|
\les
e^{-\rho_0 t}\sup_{x\in\mathbb R^n}|\underline{W}(t,x)|\les
e^{-\rho_0 t}M_1^\i
 .
$$
 Hence
$ \ds
\lim\limits_{t\to\infty}\Gamma_{0,t}\underline W(t,x)=0,
$ uniformly in $x\in\mathbb R^n.
$ 
 
 If, in addition, $\beta_2=0$, then {\bf(H3)} becomes
$$
|g(s,x,0,0,u,v)|\les \beta_1(s) ,
\qquad \forall (s,x,u,v)\in [0,\infty)\times \mathbb{R}^n\times U\times V.
$$
 Let $(u,v)\in \cU_{t,\infty}\times \cV_{t,\infty}$ be arbitrary. We write the solution of the infinite horizon BSDE \eqref{BSDE-NA} as
 $(Y_s,Z_s):=(Y_s^{t,x;u,v},Z_s^{t,x;u,v})$ for simplicity. Arguing as in Proposition  \ref{Pro-A.3} (via linearization with respect to $z$ and a Girsanov transformation),   BSDE \eqref{BSDE-NA} can be rewritten as
$$
Y_s
=Y_T+
\int_s^T \big((a_r-\rho(r))Y_r+\langle b_r,Z_r\rangle+g(r,X_r,0,0,u_r,v_r)\big)\,\md r
-\int_s^T Z_r\,\md B_r, \ \forall 0\les s\les T<\i,
$$
where
$$\ba{ll}
\ns\ds a_r:=
\begin{cases}
\dfrac{g(r,X_r,\rho(r)Y_r,Z_r,u_r,v_r)-g(r,X_r,0,Z_r,u_r,v_r)}{Y_r}, & Y_r\neq0,\\[2mm]
0,& Y_r=0,
\end{cases}\\
\ns\ds  b_r:=
\begin{cases}
\dfrac{g(r,X_r,0,Z_r,u_r,v_r)-g(r,X_r,0,0,u_r,v_r)}{|Z_r|^2}\,Z_r, & \q \ Z_r\neq0,\\[2mm]
0,& \q \  Z_r=0.
\end{cases}
\ea$$
and $
a_r\les 0,$
$
|b_r|\les L_z
$, for a.e. $r\ges t$.

After the Girsanov transform,  we can introduce a probability measure $\mathbb Q$ such that
$$\ba{ll}
\ns\ds
|Y_t^{t,x;u,v}|
=
\Big|\mathbb E_t^{\mathbb Q}\[
\int_t^\infty
\exp\(\int_t^s \big(a_r-\rho(r)\big)\,\md r\)
g(s,X_s,0,0,u_s,v_s)\,\md s
\]\Big|\\
\ns\ds\qq\qq \les
\int_t^\infty \Gamma_{t,s}\,\beta_1(s)\,\md s
\les
\int_t^\infty \beta_1(s)\,\md s  \to 0,
\qquad \mbox{as }t\to\infty ,
\ea$$
 uniformly in $(x,u,v)$.
Then,
$ \lim\limits_{t\to\i}
 \underline{W}(t,x )
=  0,
$
uniformly in $x\in\mathbb{R}^n$.

 \ms
  
\no \emph{Step 3.} We show that  $\underline{W}$ is a viscosity solution of \eqref{HJBI-NA-L}. For the supersolution, suppose that $\varphi\in C_{l,b}^{3}([0,\i)\times\mathbb{R}^n)$ and $(t_0,x_0)\in[0,\i)\times\mathbb{R}^n$ are such that $\underline{W}-\varphi$ attains its local minimum at $(t_0,x_0)$. Without loss of generality, we may assume $\underline{W}(t_0,x_0)=\varphi(t_0,x_0)$.

Using Step 1,  the continuity of $W^T$, and  a standard stability argument for local minima, we can find sequences
$\{T_k \}_{k\ges 1}$ and  $\{(t_k,x_k)\}_{k\ges 1}$  with $T_k\to\infty$, and $(t_k,x_k)\to(t_0,x_0)$ as $k\to\infty$
such that, up to a subsequence if necessary, the following hold,

 (i)  $W^{T_k}-\varphi\ges W^{T_k}(t_k,x_k)-\varphi(t_k,x_k)$   in a neighbourhood of $(t_k,x_k)$, for all $k\ges 1$;

(ii)  $W^{T_k}(t_k,x_k)\to \underline{W}(t_0,x_0)$, as  $k\to\i$.

\ss

Since $W^T$ is a viscosity supersolution to \eqref{approximating finite horizon HJBI equation}, for each $k$,  $(t_k,x_k)\in[0,T_k)\times\mathbb{R}^n$, we have
$$\frac{\partial}{\partial t}\varphi(t_k,x_k)-\rho(t_k)W^{T_k}(t_k,x_k)
 +\underline{H}\big(t_k,x_k,W^{T_k}(t_k,x_k),D\f(t_k,x_k),D^2\f (t_k,x_k)\big)\les0.
       $$
Letting $k\to\i$,  using  (ii), the continuity of $\underline{H}$ and $\rho$, we obtain
                 \begin{equation*}\label{}
  \ba{ll}
 \ns\ds\frac{\partial}{\partial t}\varphi(t_0,x_0)-\rho(t_0) \underline{W}(t_0,x_0) +\underline{H}\big(t_0,x_0,\underline{W}(t_0,x_0),D\f(t_0,x_0),
 D^2\f(t_0,x_0)\big) \les0 .
   \ea   \end{equation*}
 Hence, combined with Step 2, $\underline W$ is a viscosity supersolution of HJBI equation \eqref{HJBI-NA-L}
on $[0,\infty)\times\mathbb R^n$ in the sense of Definition
\ref{Def-viscosity-solution-NA}.  By a similar argument, one shows
that $\underline W$ is also a viscosity subsolution of \eqref{HJBI-NA-L}
on $[0,\infty)\times\mathbb R^n$. Consequently, $\underline W$ is a
viscosity solution of \eqref{HJBI-NA-L} on $[0,\infty)\times\mathbb R^n$.

 \ss

\no\emph{Step 4.}
The uniqueness  follows from the comparison theorem stated in Theorem \ref{Comparison principle} below.
  \end{proof}

  \bt\label{Comparison principle}\sl
  Assume that {\bf (A3)} and {\bf(H1)}--{\bf(H4)} hold. Let
$W_1,W_2\in C_b([0,\infty)\times\mathbb R^n)$ be, respectively, a
viscosity subsolution and a viscosity supersolution of \eqref{HJBI-NA-L}
on $[0,\infty)\times\mathbb R^n$. Then
$
W_1(t,x)\les W_2(t,x),
$ $ (t,x)\in [0,\infty)\times\mathbb{R}^n.
$

%
%
%
%

   \et

  \begin{proof}
 For $ i=1,2,$  we define
$
\widetilde W_i(t,x):=\Gamma_{0,t}\, W_i(t,x),
$ $(t,x)\in [0,\i)\times\dbR^n.$
Since $W_1,W_2\in C_b([0,\infty)\times\mathbb R^n)$ and
$\rho(t)\ges \rho_0>0$, we have
$
\lim\limits_{t\to\infty}\widetilde W_i(t,x)=0,
$ uniformly in $x\in\mathbb R^n,$ $i=1,2.
$
So, for any  $\varepsilon>0$, there exists
$T_\varepsilon>0$ such that, for all $(t,x)\in[T_\varepsilon,\infty)
\times\mathbb R^n$,
$
\widetilde W_1(t,x)\les \varepsilon,
$ $
\widetilde W_2(t,x)\ges -\varepsilon.
$
Hence
\bel{wt-W12}
\widetilde W_1(t,x)\les \widetilde W_2(t,x)+2\varepsilon,
\qquad (t,x)\in[T_\varepsilon,\infty)\times\mathbb R^n.
\ee

A direct computation shows that $\widetilde W_1$ and $\widetilde W_2$
are, respectively, a bounded viscosity subsolution and a bounded
viscosity supersolution of the following  transformed equation
\bel{3-wt-W}
\widetilde W_t+
\widetilde H(t,x,(\widetilde W,D\widetilde W,D^2\widetilde W)(t,x))=0,
\qquad (t,x)\in[0,\infty)\times\mathbb R^n,
\ee
where
$
\widetilde H(t,x,r,p,A)
:=
\Gamma_{0,t}\,
\underline H\(t,x,\frac r{\Gamma_{0,t}},
\frac p{\Gamma_{0,t}},
\frac A{\Gamma_{0,t}}\).
$
Then, $\wt W_1$  is clearly a bounded upper semicontinuous viscosity subsolution of \eqref{3-wt-W} on $[0,T_\varepsilon]\times\mathbb{R}^n$. We also note that  $\widetilde W_2+2\varepsilon$ is still a bounded
continuous  supersolution of \eqref{3-wt-W} on $[0,T_\varepsilon]\times\mathbb R^n$  due to  the non-increasing of
$r\mapsto \widetilde H(t,x,r,p,A)$ implied by {\bf (H2)}.

  Applying the finite horizon  comparison principle  on
$[0,T_\varepsilon]\times\mathbb R^n$ (refer to \cite{Buckdahn-Li-2008}), we obtain
$
\widetilde W_1(t,x)
\les
\widetilde W_2(t,x)+2\varepsilon,
$ $ (t,x)\in[0,T_\varepsilon]\times\mathbb R^n.
$
Combining this with \eqref{wt-W12}, we get
$$
\widetilde W_1(t,x)
\les
\widetilde W_2(t,x)+2\varepsilon,
\qquad (t,x)\in[0,\infty)\times\mathbb R^n.
$$
Letting $\varepsilon\downarrow0$ yields
$
\widetilde W_1(t,x)\les \widetilde W_2(t,x),
$ $ (t,x)\in[0,\infty)\times\mathbb R^n.
$
Since $\Gamma_{0,t}>0$, it follows that
$
W_1(t,x)\les W_2(t,x),
$ $ (t,x)\in[0,\infty)\times\mathbb R^n.
$
This completes the proof.
  \end{proof}

\begin{remark}\sl
    The proof of Theorem  \ref{Comparison principle} may alternatively be interpreted
through the discounted time transformation
$
\tau=\Gamma_{0,t}.
$
For $W\in C_b([0,\infty)\times\mathbb R^n)$, the transformed function
$$
\widehat W(\tau,x):=\tau W(t(\tau),x)
$$
satisfies $\widehat W(\tau,x)\to0$ uniformly in $x$ as
$\tau\downarrow0$. Hence the boundary condition at infinity becomes a
boundary condition at the finite endpoint $\tau=0$. One can then apply
the finite horizon comparison principle on $[\delta,1]\times\mathbb R^n$
and let $\delta\downarrow0$. This gives the same comparison result as
Theorem \ref{Comparison principle}.

 \er

  \begin{corollary}\label{Cor-upper}\sl
Under assumptions {\bf (A3)} and  {\bf(H1)}--{\bf(H4)}, all the conclusions of Theorems \ref{Th-NA} and \ref{Comparison principle} remain valid for the upper value function $\overline{W}$. More precisely, $\overline W$ is deterministic and is the
unique viscosity solution in  $C_b([0,\infty)\times\mathbb R^n)$ of the following upper HJBI equation
$$\left\{\2n
\ba{ll}
\ns\ds 
\frac{\partial}{\partial t}\overline{W}(t,x)
-\rho(t)\overline{W}(t,x)
+\overline{H}\bigl(t,x,(\overline{W},D\overline{W},D^2\overline{W})(t,x)\bigr)
=0,
\quad
(t,x)\in [0,\infty)\times \mathbb{R}^n,\\
\ns\ds \lim\limits_{T\to\infty}\Gamma_{0,T}\overline W(T,x)=0,
  \mbox{ uniformly in } x\in\mathbb R^n,\\
\ea\right.$$
where,  for 
$(t,x,y,p,A)\in [0,\infty)\times\mathbb R^n\times\mathbb R
\times\mathbb R^n\times\mathbb S^n$,     
$$
\overline  H(t,x,y,p,A)
:=
\inf_{v\in V}\sup_{u\in U}
\Big\{
\frac12\operatorname{tr}
\big(\sigma\sigma^\top(t,x,u,v)A\big)
+b(t,x,u,v)\cdot p
+g\big(t,x,\rho(t)y,p\sigma(t,x,u,v),u,v\big)
\Big\}.
$$  
Moreover, if, in addition, $\beta_2=0$,  then the above boundary condition can   be
strengthened to
$
\lim\limits_{T\to\infty}\overline W(T,x)=0,
$ uniformly in $x\in\mathbb R^n.
$
 
\end{corollary}

\begin{corollary}\sl
Assume further that the Isaacs condition holds, namely,
$$
\underline{H}(t,x,y,p,X)=\overline{H}(t,x,y,p,X),
\quad
(t,x,y,p,X)\in [0,\infty)\times\mathbb{R}^n\times\mathbb{R}\times\mathbb{R}^n
\times\cS^n.
$$
Then the stochastic differential game has a value. More precisely,
$
\underline{W}(t,x)=\overline{W}(t,x)=:W(t,x),
$
$(t,x)\in [0,\infty)\times \mathbb{R}^n,
$
where $W$ is the value function of the game.

\end{corollary}

Finally, we show that, when the coefficients are time-homogeneous and the
discount factor is constant, the present non-autonomous formulation reduces to,
and is consistent with, the classical stationary infinite horizon case.

\begin{corollary}
\label{cor-time-homogeneous}\sl
Assume that the coefficients in \eqref{state-NA} and \eqref{BSDE-NA} are
independent of the time variable and that the discount factor is constant,
namely,
\bel{Indep-t}
b=b(x,u,v),\quad
\sigma=\sigma(x,u,v),\quad
g=g(x,y,z,u,v),\quad
\rho(t)\equiv \rho>0.
\ee
Then the following assertions hold.

{\rm(i)} The lower value function $
\underline W $  introduced in \eqref{NA-uW} does not  depend 
on the  time variable, denote by  $\underline w$. 
  Moreover, $\underline w$ is the unique bounded viscosity solution of the
stationary lower HJBI equation
\bel{S-LHJBI}
-\rho  w(x)
+
\sup_{u\in U}\inf_{v\in V}
\big\{
\mathcal L^{u,v}  w(x)
+
g\bigl(x,\rho w(x),
D  w(x)\sigma(x,u,v),u,v\bigr)
\big\}
=0,
\qquad x\in\mathbb R^n,
\ee
where, for $(u,v)\in U\times V$,  
 $\ds
\mathcal L^{u,v}\phi(x)
:=
\frac12\operatorname{tr}
\bigl(\sigma\sigma^\top(x,u,v)D^2\phi(x)\bigr)
+
b(x,u,v)\cdot D\phi(x).
$ 
The comparison principle for  viscosity subsolutions and supersolutions
of \eqref{S-LHJBI} also holds in  $C_b(\mathbb R^n)$. 

{\rm(ii)}The upper value function $
\overline W$ in \eqref{NA-oW}
is independent of the time variable, denoted by $\overline w$. 
Moreover, $\overline w$ is the unique bounded viscosity solution of the
stationary upper HJBI equation
\bel{S-UHJBI}
-\rho  w(x)
+
\inf_{v\in V}\sup_{u\in U}
\big\{
\mathcal L^{u,v} w(x)
+
g\bigl(x,\rho  w(x),
D w(x)\sigma(x,u,v),u,v\bigr)
\big\}
=0,
\qquad x\in\mathbb R^n.
\ee
The comparison principle holds for   viscosity
subsolutions and supersolutions of  \eqref{S-UHJBI} in  $C_b(\mathbb R^n)$.
\end{corollary}
 
\begin{proof}
We only prove assertion (i). Assertion (ii) follows in the same way.
Under condition \eqref{Indep-t}, the lower HJBI equation \eqref{HJBI-NA-L}  reduces to
\begin{equation}\label{HJBI-time-homogeneous-parabolic}
\partial_t W(t,x)-\rho W(t,x)
+\sup_{u\in U}\inf_{v\in V}
\left\{
\mathcal L^{u,v}W(t,x)
+g\bigl(x,\rho W(t,x),DW(t,x)\sigma(x,u,v),u,v\bigr)
\right\}=0 .
\end{equation}
\emph{ Step 1.} We first show that $\underline W$ is independent of the   time variable. To this end, for
each $h\ges0$, define
$$
\underline W^h(t,x):=\underline W(t+h,x),
\qquad (t,x)\in[0,\infty)\times\mathbb R^n.
$$
Then
\bel{W-h-BC}
\sup_{x\in\mathbb R^n}
\big|e^{-\rho T}\underline W^h(T,x)\big|
\les
e^{\rho h}
\sup_{x\in\mathbb R^n}
\big|e^{-\rho(T+h)}\underline W(T+h,x)\big|
\to 0,
\quad \mbox{as } T\to\infty.
\ee
We claim that $\underline W^h$ is again a  viscosity solution of
\eqref{HJBI-time-homogeneous-parabolic} in $C_b([0,\i)\times \dbR^n)$. 
Indeed,   $\underline W^h\in C_b([0,\i)\times \dbR^n)$ is obvious. We only verify the viscosity inequalities. Let
$\varphi\in C^{1,2}([0,\infty)\times\mathbb R^n)$ and suppose that
$\underline W^h-\varphi$ attains a local maximum at $(t_0,x_0)$. Define, in a
neighbourhood of $(t_0+h,x_0)$,
$ 
\psi(s,x):=\varphi(s-h,x).
$ 
Then $\underline W-\psi$ attains a local maximum at $(t_0+h,x_0)$. Since
$\underline W$ is a viscosity subsolution of \eqref{HJBI-time-homogeneous-parabolic},
we have
$$\ba{ll}
\ns\ds 
\partial_s\psi(t_0+h,x_0)
-\rho \underline W(t_0+h,x_0)\\
\ns\ds 
+\sup_{u\in U}\inf_{v\in V}
\left\{
\mathcal L^{u,v}\psi(t_0+h,\cdot)(x_0)
+g\bigl(x_0,\rho\underline W(t_0+h,x_0),
D\psi(t_0+h,x_0)\sigma(x_0,u,v),u,v\bigr)
\right\}
\ges 0 .
\ea$$
By the definition of $\psi$ and $\underline W^h$, this is exactly
$$\ba{ll}
\ns\ds 
\partial_t\varphi(t_0,x_0)
-\rho \underline W^h(t_0,x_0)\\
\ns\ds 
+\sup_{u\in U}\inf_{v\in V}
\big\{
\mathcal L^{u,v}\varphi(t_0,\cdot)(x_0)
+g\bigl(x_0,\rho\underline W^h(t_0,x_0),
D\varphi(t_0,x_0)\sigma(x_0,u,v),u,v\bigr)
\big\}
\ges 0 .
\ea$$
Hence, combined with the boundary condition \eqref{W-h-BC},  $\underline W^h$ is a viscosity subsolution of \eqref{HJBI-time-homogeneous-parabolic}.
Similarly, we can prove the supersolution inequality. Therefore
$\underline W^h$ is a  viscosity solution of
\eqref{HJBI-time-homogeneous-parabolic} in  $C_b([0,\i)\times \dbR^n)$. 

By the uniqueness of the viscosity solution of \eqref{HJBI-time-homogeneous-parabolic} established in Theorem \ref{Comparison principle}, we get
$ 
\underline W^h(t,x)=\underline W(t,x),
$ for all $ t,$ $ h\ges0,\ x\in\mathbb R^n.
$  
Thus $\underline W$ is independent of $t$. 

\ms

\emph{Step 2.}
We denote
$ 
\underline w(x):=\underline W(0,x),$ $ x\in\mathbb R^n.
$ 
Since $\underline W(t,x)=\underline w(x)$, it's easy to check the parabolic equation
\eqref{HJBI-time-homogeneous-parabolic} reduces, in the viscosity sense, to
the stationary lower HJBI equation \eqref{S-LHJBI}.
It remains to prove the comparison principle for  \eqref{S-LHJBI}. 
Let $w_1,w_2\in C_b(\mathbb R^n)$ be, respectively, a viscosity subsolution
and a viscosity supersolution of \eqref{S-LHJBI}. Define
$$
W_i(t,x):=w_i(x),
\qquad (t,x)\in[0,\infty)\times\mathbb R^n,\quad i=1,2.
$$
Since
$w_1,w_2$ are bounded,
$ \ds
\lim_{T\to\infty}e^{-\rho T}W_i(T,x)=0,
$  uniformly in $x\in\mathbb R^n,$ $ i=1,2.
$ Moreover,  $W_1$ and $W_2$ are, respectively, a viscosity subsolution and a
viscosity supersolution of \eqref{HJBI-time-homogeneous-parabolic}. 
 Applying the comparison
principle  of \eqref{HJBI-time-homogeneous-parabolic} in Theorem \ref{Comparison principle}, we obtain
$$
W_1(t,x)\les W_2(t,x),
\qquad (t,x)\in[0,\infty)\times\mathbb R^n.
$$
Taking $t=0$ gives
$ 
w_1(x)\les w_2(x),
$ $ x\in\mathbb R^n.
$ 
Hence the comparison principle holds for \eqref{S-LHJBI}
in $C_b(\mathbb R^n)$, and uniqueness follows immediately.
\end{proof}

\begin{remark}\label{Re-stationarity-PDE-viewpoint}\sl
The stationarity of autonomous infinite horizon stochastic recursive control and  game problems  was indicated  in \cite{Li-Zhao-2019, Buckdahn-Li-Zhao-2021, Luo-Li-Wei-2025} by a
probabilistic shift argument. The preceding corollary gives a complementary
PDE-based derivation: the time-shifted function
$\underline W^h(t,x):=\underline W(t+h,x)$ solves the same non-autonomous HJBI
equation and satisfies the same boundary condition at infinity; by uniqueness,
$\underline W^h=\underline W$. Hence the value function is independent of the
initial time, and the stationary HJBI equation is recovered here as a
consequence of the non-autonomous PDE uniqueness result.
\end{remark}
 
 \section*{Acknowledgements}

The authors thank Professor Jiongmin Yong for helpful discussions and valuable
suggestions.


\begin{thebibliography}{99}

\bibitem{Basco-Frankowska-2019}
V. Basco and H. Frankowska,
Lipschitz continuity of the value function for the infinite horizon optimal
control problem under state constraints,
in \emph{Trends in Control Theory and Partial Differential Equations},
F. Alabau-Boussouira et al., eds., Springer INdAM Series 32,
Springer, Cham, 2019, 17--39.

\bibitem{Baumeister-Leitao-Silva-2007}
J. Baumeister, A. Leit{\~a}o, and G. N. Silva,
On the value function for nonautonomous optimal control problems with infinite
horizon,
\emph{Systems Control Lett.}, 56 (2007), 188--196.

\bibitem{Bielecki-Rutkowski-2004}
T. R. Bielecki and M. Rutkowski,
\emph{Credit Risk: Modeling, Valuation and Hedging},
Springer Finance, Springer, Berlin, 2004.

\bibitem{Briand-Confortola-2008}
P. Briand and F. Confortola,
Quadratic BSDEs with random terminal time and elliptic PDEs in infinite
dimension,
\emph{Electron. J. Probab.}, 13 (2008), 1529--1561.

\bibitem{Briand-Hu-1998}
P. Briand and Y. Hu,
Stability of BSDEs with random terminal time and homogenization of semilinear
elliptic PDEs,
\emph{J. Funct. Anal.}, 155 (1998), 455--494.

\bibitem{Buckdahn-Hu-Li-2011}
R. Buckdahn, Y. Hu, and J. Li,
Stochastic representation for solutions of Isaacs' type integral-partial
differential equations,
\emph{Stochastic Process. Appl.}, 121 (2011), 2715--2750.

\bibitem{Buckdahn-Li-2008}
R. Buckdahn and J. Li,
Stochastic differential games and viscosity solutions of
Hamilton--Jacobi--Bellman--Isaacs equations,
\emph{SIAM J. Control Optim.}, 47 (2008), 444--475.

\bibitem{Buckdahn-Li-Zhao-2021}
R. Buckdahn, J. Li, and N. Zhao,
Representation of limit values for nonexpansive stochastic differential games,
\emph{J. Differential Equations}, 276 (2021), 187--227.

\bibitem{Carlson-Haurie-Leizarowitz-1991}
D. A. Carlson, A. B. Haurie, and A. Leizarowitz,
\emph{Infinite Horizon Optimal Control: Deterministic and Stochastic Systems},
2nd ed., Springer, Berlin, 1991.

\bibitem{Crandall-Ishii-Lions-1992}
M. G. Crandall, H. Ishii, and P.-L. Lions,
User's guide to viscosity solutions of second order partial differential
equations,
\emph{Bull. Amer. Math. Soc.}, 27 (1992), 1--67.

\bibitem{Darling-Pardoux-1997}
R. W. R. Darling and {\'E}. Pardoux,
Backwards SDE with random terminal time and applications to semilinear elliptic
PDE,
\emph{Ann. Probab.}, 25 (1997), 1135--1159.

\bibitem{Dockner-Jorgensen-Long-Sorger-2000}
E. J. Dockner, S. J{\o}rgensen, N. V. Long, and G. Sorger,
\emph{Differential Games in Economics and Management Science},
Cambridge University Press, Cambridge, 2000.

\bibitem{Fleming-McEneaney-1995}
W. H. Fleming and W. M. McEneaney,
Risk-sensitive control on an infinite time horizon,
\emph{SIAM J. Control Optim.}, 33 (1995), 1881--1915.

\bibitem{Fleming-Sheu-2002}
W. H. Fleming and S.-J. Sheu,
Risk-sensitive control and an optimal investment model. II,
\emph{Ann. Appl. Probab.}, 12 (2002), 730--767.

\bibitem{Fleming-Soner-2006}
W. H. Fleming and H. M. Soner,
\emph{Controlled Markov Processes and Viscosity Solutions},
2nd ed., Stochastic Modelling and Applied Probability 25,
Springer, New York, 2006.

\bibitem{Fleming-Souganidis-1989}
W. H. Fleming and P. E. Souganidis,
On the existence of value functions of two-player, zero-sum stochastic
differential games,
\emph{Indiana Univ. Math. J.}, 38 (1989), 293--314.

\bibitem{Hu-Tessitore-2007}
Y. Hu and G. Tessitore,
BSDE on an infinite horizon and elliptic PDEs in infinite dimension,
\emph{NoDEA Nonlinear Differential Equations Appl.}, 14 (2007), 825--846.

\bibitem{Krylov-1980}
N. V. Krylov,
\emph{Controlled Diffusion Processes},
Applications of Mathematics 14, Springer, New York, 1980.

\bibitem{Li-Li-Wei-2021}
J. Li, W. Li, and Q. Wei,
Probabilistic interpretation of a system of coupled
Hamilton--Jacobi--Bellman--Isaacs equations,
\emph{ESAIM Control Optim. Calc. Var.}, 27 (2021), S17.

\bibitem{Li-Peng-2009}
J. Li and S. Peng,
Stochastic optimization theory of backward stochastic differential equations
with jumps and viscosity solutions of Hamilton--Jacobi--Bellman equations,
\emph{Nonlinear Anal.}, 70 (2009), 1776--1796.

\bibitem{Li-Wei-2014}
J. Li and Q. Wei,
Optimal control problems of fully coupled FBSDEs and viscosity solutions of
Hamilton--Jacobi--Bellman equations,
\emph{SIAM J. Control Optim.}, 52 (2014), 1622--1662.

\bibitem{Li-Wei-2015}
J. Li and Q. Wei,
Stochastic differential games for fully coupled FBSDEs with jumps,
\emph{Appl. Math. Optim.}, 71 (2015), 411--448.

\bibitem{Li-Zhao-2019}
J. Li and N. Zhao,
Representation of asymptotic values for nonexpansive stochastic control
systems,
\emph{Stochastic Process. Appl.}, 129 (2019), 634--673.

\bibitem{Lin-Ren-Touzi-Yang-2020}
Y. Lin, Z. Ren, N. Touzi, and J. Yang,
Second order backward SDE with random terminal time,
\emph{Electron. J. Probab.}, 25 (2020), article 99, 1--43.

\bibitem{Luo-Li-Wei-2025}
S. Luo, X. Li, and Q. Wei,
Infinite horizon stochastic recursive control problems with jumps: dynamic
programming and stochastic verification theorems,
\emph{SIAM J. Control Optim.}, 63 (2025), 796--821.

\bibitem{Nisio-1988}
M. Nisio,
Stochastic differential games and viscosity solutions of Isaacs equations,
\emph{Nagoya Math. J.}, 110 (1988), 163--184.

\bibitem{Pardoux-1998}
{\'E}. Pardoux,
Backward stochastic differential equations and viscosity solutions of systems
of semilinear parabolic and elliptic PDEs of second order,
in \emph{Stochastic Analysis and Related Topics VI},
Birkh{\"a}user, Boston, 1998, 79--127.

\bibitem{Pardoux-Peng-1990}
{\'E}. Pardoux and S. Peng,
Adapted solution of a backward stochastic differential equation,
\emph{Systems Control Lett.}, 14 (1990), 55--61.

\bibitem{Peng-1991}
S. Peng,
Probabilistic interpretation for systems of quasilinear parabolic partial
differential equations,
\emph{Stochastics Stochastics Rep.}, 37 (1991), 61--74.

\bibitem{Peng-1997}
S. Peng,
Backward stochastic differential equations--stochastic optimization theory and
viscosity solutions of HJB equations,
in \emph{Topics on Stochastic Analysis},
J. Yan, S. Peng, S. Fang, and L. Wu, eds.,
Science Press, Beijing, 1997, 85--138 (in Chinese).

\bibitem{Peskir-Shiryaev-2006}
G. Peskir and A. Shiryaev,
\emph{Optimal Stopping and Free-Boundary Problems},
Lectures in Mathematics, ETH Z{\"u}rich, Birkh{\"a}user, Basel, 2006.

\bibitem{Pham-2009}
H. Pham,
\emph{Continuous-Time Stochastic Control and Optimization with Financial
Applications},
Stochastic Modelling and Applied Probability 61,
Springer, Berlin, 2009.

\bibitem{Royer-2004}
M. Royer,
BSDEs with a random terminal time driven by a monotone generator and their
links with PDEs,
\emph{Stochastics Stochastics Rep.}, 76 (2004), 281--307.

\bibitem{Swiech-1996}
A. {\'S}wi\c{e}ch,
Another approach to the existence of value functions of stochastic differential
games,
\emph{J. Math. Anal. Appl.}, 204 (1996), 884--897.

\bibitem{Wei-Yu-2021}
Q. Wei and Z. Yu,
Infinite horizon forward-backward SDEs and open-loop optimal controls for
stochastic linear-quadratic problems with random coefficients,
\emph{SIAM J. Control Optim.}, 59 (2021), 2594--2623.

\bibitem{Wei-Yong-2025}
Q. Wei and J. Yong,
A time-inconsistent stochastic optimal control problem in an infinite time
horizon,
arXiv:2509.14495v1, 2025.

\bibitem{Yong-Zhou-1999}
J. Yong and X. Zhou,
\emph{Stochastic Controls: Hamiltonian Systems and HJB Equations},
Springer, New York, 1999.

\end{thebibliography}
\end{document}